\documentclass[11pt]{amsart}
\usepackage{amsmath}
\usepackage{amssymb}
\usepackage{amsthm}
\usepackage{latexsym}
\usepackage{epsfig}
\usepackage{a4wide}    
\usepackage{tikz}

\newcommand{\N}{\ensuremath{\mathbb N}}
\newcommand{\Z}{\ensuremath{\mathbb Z}}
\newcommand{\R}{\ensuremath{\mathbb R}}

\newcommand{\Q}{\ensuremath{\mathbb Q}}

\newcommand{\X}{\ensuremath{\mathbb X}}
\newcommand{\Xb}{\ensuremath{\mathbf X}}

\newcommand{\D}{\ensuremath{{\mathcal D}}}

\newcommand{\T}{\ensuremath{{\mathcal T}}}
\newcommand{\A}{\ensuremath{{\mathcal A}}}

\newcommand{\Xc}{\ensuremath{{\mathcal X}}}

\renewcommand{\rho}{\varrho}
\renewcommand{\phi}{\varphi}
\newcommand{\RR}{\mathfrak R}

\DeclareMathOperator{\cl}{cl}
\DeclareMathOperator{\id}{id}

\newtheorem{thm}{Theorem}[section]
\newtheorem{defi}[thm]{Definition}

\newtheorem{cor}[thm]{Corollary}

\newtheorem{prop}[thm]{Proposition}

\theoremstyle{definition}
\newtheorem{rem}[thm]{Remark}
\newtheorem{example}[thm]{Example}
\begin{document}

\title{Selfdual substitutions in dimension one}

\author[V.~Berth\'e]{Val\'erie Berth\'e}
\address{LIAFA - CNRS   UMR 7089, Universit\'e Paris Diderot -- Paris 7,
Case 7014, 75205 Paris Cedex 13, France} 
\email{berthe@liafa.jussieu.fr}
\urladdr{http://www.lirmm.fr/\textasciitilde  berthe}

\author[D.~Frettl\"oh]{Dirk Frettl\"oh}
\address{Institut f\"ur Mathematik und Informatik, Universit\"at Greifswald /
Institut f\"ur Mathematik, FU Berlin}
\email{dirk.frettloeh@udo.edu}
\urladdr{http://page.mi.fu-berlin.de/frettloe/}

\author[V.~Sirvent]{Victor Sirvent}
\address{ Departamento de Matem\'aticas, Universidad Sim\'on
Bol\'{\i}var, Apartado 89000, Caracas 1086-A, Venezuela}
\email{vsirvent@usb.ve}
\begin{abstract} 
There are several notions of the `dual' of a word/tile substitution. 
We show that the most common ones are equivalent for substitutions
in dimension one, where we restrict ourselves to the case of two
letters/tiles. Furthermore, we obtain necessary and sufficient
arithmetic conditions for substitutions being selfdual in this case. 
Since many connections between the different notions of word/tile
substitution are discussed,  this paper may also serve as a survey
paper on this topic.
\end{abstract} 
\date{\today}

\maketitle

% \tableofcontents

\section{Introduction}

Substitutions are simple but powerful tools to generate a large
number of nonperiodic structures with a high degree of order. Examples
include infinite words (e.g., the Thue Morse sequence, see \cite{Que,fogg}), infinite
tilings (e.g., Penrose tilings, see  \cite{Rob,So}) and discrete point sets (e.g., models
of atomic positions in quasicrystals, see \cite{baaqc,moo}). Here we consider several
instances of the concept of substitutions:
\begin{itemize}
\item[(a)] word substitutions 
\item[(b)] nonnegative endomorphisms of the free group $F_2 = \langle a, b \rangle$ 
\item[(c)] tile-substitutions
\item[(d)] dual maps of substitutions 
\end{itemize}
Each of the concepts above gives rise to the concept of a dual
substitution. Our first goal is to show the full equivalence of the
distinct notions of dual substitution with respect to the concepts
above for the case of two letters,   or   of two tiles in $\R^1$. Thus
we will exclusively study  substitutions on two letters,
and for the sake of clarity, we will define every term for this
special case only. From here on, we will consider only this case,
whether or not this is stated explicitly.

Let us mention that there is a wealth of results for
word  substitutions on two letters, thus for tile-substitutions 
in $\R^1$ with two tiles. For a start see \cite{Que,Lot2,fogg} and references
therein. The following theorem lists some interesting results which
emerged in the work of many authors during the last decades.  We recall that  a word  substitution   is   a morphism of the free monoid.
A  word substitution is said to be primitive  if there exists  a  power $n$ such that
the image of    any letter by $\sigma ^n$ contains all the letters of the alphabet.

\begin{thm} \label{thm:intro}
Let $\sigma$ be a primitive word substitution on two letters. 
Then the following are equivalent:
\begin{enumerate}
\item Each bi-infinite word $u$ generated by $\sigma$ is 
  {\em Sturmian}, i.e., $u$ contains exactly $n+1$ different words of
  length $n$ for all $n \in \N$.  
\item The endomorphism $\sigma: F_2 \to F_2$ is invertible, i.e.,
  $\sigma \in \text{Aut}(F_2)$. 
\item Each tiling generated by $\sigma$ is a cut and project tiling
  whose window is an interval.  
\item There is $k \ge 1$ such that the substitution $\sigma^k$ is
  conjugate to $L^{a_1}  E {L}^{a_2} E  L^{a_3} \cdots
E{L}^{a_m}$, where $a_i \in \N \cup \{0\}$, $L: \, a \to a,
 \, b \to ab$, $E: \, a \to b, \, b \to a$. 
\end{enumerate}
\end{thm}

The equivalence of (2) and (4) follows from \cite{wen}. The
equivalence of (2) and (3) appears in \cite{lamb}, and relies on
earlier results, see references in \cite{lamb,beir}. The equivalence 
of (1) and (3) is due to \cite{morse,coven}. 
For the equivalence of (1) and (4), see \cite{Seebold98}.
For more details,  see   \cite{Lot2,fogg}.
%or Appendix \ref{subsec:sturm} below.

This theorem links automorphisms of the free group $F_2$ with
bi-infinite words over two letters, and  with tilings of the line by
intervals of two distinct lengths. Our motivation was to work out 
this correspondence. 

The paper is organised as follows: Section \ref{sec:word} provides the 
necessary definitions and facts about word substitutions, and about
(inverses of) substitutions as automorphisms of $F_2$. Section
\ref{sec:tilings} contains the relevant definitions and facts about
(duals of) substitution tilings and cut and project tilings.  The very close relation of the natural decomposition
method of \cite{sirw} to duals of cut and project tilings  via so-called Rauzy
fractals is given in
Section \ref{sec:dcm}. Section
\ref{sec:dual} introduces generalised substitutions and
 the essential definitions and facts for
strand spaces and dual maps of substitutions. Any reader familiar with
any of this may skip the corresponding section(s). Section
\ref{sec:relations} explains how the distinct concepts of
dual substitutions are related. The first main result, Theorem
\ref{thm:inv}, shows that they all are equivalent, in the sense,
that they generate equivalent hulls. (For the definition of the hull
of a word substitution
see Equation \eqref{eq:hull} below. For a precise definition of
equivalence of hulls, see Section \ref{sec:relations}.) In Section
\ref{sec:selfdual} we give several necessary and sufficient conditions
for a substitution to be selfdual, namely, Theorems \ref{thm:master}
and \ref{thm:sdfreq}. These results yield a complete classification of
selfdual word substitutions on two letters. 
%Some of the more technical computations
%underlying the proof of Theorem \ref{thm:sdfreq} are given in Appendix
%\ref{sec:arithm}. 

\section {Word substitutions} \label{sec:word}

\subsection{Basic definitions and notation} \label{subsec:basic}
Let $\A=\{ a,b\}$ be a finite alphabet, let $\A^{\ast}$ be the set of all finite words over $\A$,
and let $\A^{\Z} = \{(u_i)_{i \in \Z}  \mid u_i \in \A \}$ be the set
of all bi-infinite words over $\A$.    A   substitution $\sigma$ is a map
from $\A$ to $\A^{\ast}\setminus \{ \varepsilon \}$; it is  called {\em non-erasing}  if  it
 maps $\A$ to $\A^{\ast}\setminus \{ \varepsilon \}$, where
$\varepsilon$ denotes the empty word. We will exclusively consider here non-erasing substitutions. By concatenation, a substitution
extends to a map from $\A^{\ast}$ to $\A^{\ast}$, and to a map from 
$\A^{\Z}$ to $\A^{\Z}$.  
Thus a substitution can be iterated. For
instance, consider the Fibonacci substitution
\begin{equation} \label{eq:fib1}
\sigma: \A \to \A^{\ast}, \quad \sigma(a)=ab, \; \sigma(b)=a. 
\end{equation}
We abbreviate this long notation by $\sigma: a \to ab, b \to a$. Then,
$\sigma(a) = ab$, $\sigma^2(a)=\sigma(\sigma(a))=\sigma(ab) =
\sigma(a) \sigma(b) = aba$, $\sigma^3(a)=abaab$, and so on.
In order to rule out certain non-interesting cases, for instance
$E\colon  a \mapsto b, b\mapsto a$, the following definition is useful. 

\begin{defi} \label{def:prim}
 Let $\sigma$ be a   substitution   defined on ${\mathcal A}$.  The
 {\em substitution matrix}  $M_\sigma$
 associated with $\sigma$ is defined  as  its Abelianisation, i.e.,
 \[ \ M_{\sigma} =(|\sigma(j)|_i)_{1 \le i,j \le 2}  ,\]   
 where $|w|_i$  stands  for the number of    occurrences of the letter
 $i$ in $w$.\\
 A substitution $\sigma$ on $\A$ is called {\em primitive}, if its  substitution matrix is primitive, that is, if 
 $(M_{\sigma})^n >0$ for some $n \in \N$. \\
A substitution $\sigma$ is called {\em unimodular}, if its
substitution matrix $M_{\sigma}$ has determinant $1$ or $-1$.
\end{defi}

In the sequel, we will consider unimodular primitive substitutions
only. The requirement of primitivity is a common one, essentially this
rules out some pathological cases. The requirement of unimodularity is
a restriction which we use because we will focus on the case where the
substitution is invertible (see Section \ref{subsec:invert} for a definition).

We will  consider the set   of {\em all} bi-infinite
words which are `legal' with respect to a substitution $\sigma$.  One  reason for this choice
of working with bi-infinite words (instead of infinite words)  comes from  the   connections  
with  tile-substitutions that we   stress in Section~\ref{sec:tilings}.
Thus we define the {\em hull} $\Xc_{\sigma}$ of $\sigma$ as follows. 
\begin{equation} \label{eq:hull}
\Xc_{\sigma} = \{ u  \in \A^{\mathbb Z}  \mid \text{each subword of }u\text{ 
is a subword of}\; \sigma^n(a)\, \text{or}\; \sigma^n(b)\, \text{for
some}\, n\}. 
\end{equation}   
Subwords of a word are commonly also called `factors' of that word. 
\begin{prop}\cite{Que,fogg} \label{prop:ssn} 
For any primitive substitution $\sigma$ and for all $n \ge 1$ holds:
$\Xc_{\sigma^n}=\Xc_{\sigma}$.
\end{prop}

There is a further notion of a hull, namely, the hull of a bi-infinite
word. In our framework (substitutions are assumed to be primitive) the two notions coincide. The hull $\Xc_u$ of a
bi-infinite word $u$ is defined as the closure of the orbit $\{ S^k u
\mid k \in \Z \}$ in the obvious topology, see \cite{Que,fogg}. Here
$S$ denotes the shift operator $S  u = S (u_i)_{i \in \Z} =
(u_{i+1})_{i \in \Z}$,    which yields $S  ^{-1} u = 
(u_{i-1})_{i \in \Z}$.

\begin{prop} \label{uinxc}  
Let $\sigma$ be a primitive substitution. Then, for each $u \in
\Xc_{\sigma}$, $\Xc_{\sigma} = \Xc_u$. In particular, $\Xc_{\sigma}$ is
determined by each $u \in \Xc_{\sigma}$ uniquely.
\end{prop}

Let $\sigma$ be a primitive  substitution. By the Perron-Frobenius
theorem,  its  substitution matrix $M_{\sigma}$ has a dominant
real eigenvalue $\lambda >1$, which has a  positive
eigenvector (i.e.,  an eigenvector having  all its components $>0$).  We call $\lambda$ the {\em  inflation factor} of
$\sigma$ (see Section \ref{subsec:substil} for a justification of this term).
If $(1-\alpha,\alpha)$ is the normed eigenvector  of $\lambda$ ($0 <
\alpha <1$), that is, the vector  whose sum of coordinates equals  $1$, then $\alpha$ is called the {\em  frequency} of the
substitution. Indeed, it is not hard to see that every bi-infinite word
in $\Xc_\sigma$ has well defined letter frequencies $1-\alpha$ and
$\alpha$, see for instance \cite{Que}.   Left eigenvectors  associated with
the inflation  factor $\lambda$ will also play  an important role here.  We will  work  with 
the  positive  left eigenvector $v_{\lambda}=(1, l_{\lambda})$ with first coordinate $1$.

In all that follows,    $\sigma(i)_k$ stands for the $k$-th letter of
$\sigma(i)$,  $|\sigma(i)|$ is the length of
$\sigma(i)$ and $\sigma(i)[k-1]$
 is
the prefix of length $k-1$ of $\sigma(i)$  (with  $\sigma(i)_0$  being    the empty word), i.e., 
\[\sigma(i)=\sigma(i)_1\cdots \sigma(i)_{|\sigma(i)|}=
\sigma(i)[k-1] \sigma(i)_k \sigma(i)_{k+1}\cdots \sigma(i)_{|\sigma(i)|} .\]

We denote by 
$A$   the  {\em Abelianisation map }
from ${\mathcal A} ^*$ to $\Z^2$: if $w$ is a   word  in 
${\mathcal A}^*$, then $A(w)$ is the vector
that counts the number of occurrences of each letter in $w$, i.e.,
$A\colon{\mathcal A} ^*\to\Z^2, \  w\mapsto (|w|_{a},|w|_{b}).$
There is an obvious commutative diagram, where  $M_{\sigma}$  stands
for  the  substitution  matrix of  
$\sigma$:

\[\begin{array}{ccc} {\mathcal A} ^{\ast} &
\stackrel{\sigma}{\longrightarrow}
& {\mathcal A}^{\ast} \\
A \downarrow  && \downarrow A \\
{\Z}^2 & \stackrel{ M _{\sigma}}{\longrightarrow} &{\Z}^2. \\
\end{array}\]

\subsection{Substitutions as endomorphisms of $F_2$}\label{subsec:invert}
Naturally, any word substitution on $\A$ gives rise to an endomorphism
$\sigma$ of $F_2$, the free group on two letters. (Note that not
every endomorphism of $F_2$ gives rise to a proper word substitution,
consider for instance $\sigma: \, a \to ab^{-1}, \, b \to b^{-1}$.) If
this endomorphism $\sigma$ happens to be an automorphism, that is,
$\sigma \in$ Aut($F_2$), then $\sigma$ is said to be {\em invertible}.  
It is well-known that the set of invertible word substitutions on a
two-letter alphabet is a  finitely generated monoid, with one set of generators  being
\begin{equation}\label{eq:generators}
E\colon a  \mapsto b, \ b \mapsto a, \   L\colon a \mapsto a,   \ b
 \mapsto ab, \   \tilde{L} \colon a\mapsto a,  \  b \mapsto ba.
 \end{equation}
For references, see  \cite{wen} and Chap.\ 2  in  \cite{Lot2}.

In the sequel, we want to consider whether two substitutions generate
the same sequences, that is, the same hull. 
Recall that an automorphism $\gamma$ is called {\em inner}
automorphism, if  there   exists  $w \in F_2$ such  that  $\gamma(x) = w x w^{-1}$  for every $x  \in F_2$. We let 
$\gamma_w$  denote this inner automorphism.

\begin{defi} \label{eq:defconj} 
We say that two given substitutions  $\sigma$ and $\rho$
are {\it conjugate}, if $\sigma = \gamma_w \circ \rho$
for some $w \in F_2$. In this case, we will write $\sigma \sim \rho$.
\end{defi}
\begin{rem}
We use here the term `conjugate' but  a more precise   terminology would be the following:
  $\sigma$ and $\rho$  belong to the same outer class, or else, 
$\rho$ is obtained  from $\sigma$ by   action  of   an  inner automorphism (also called conjugation).
For convenience, we take the freedom to use the short
version.   Note that  if   $\sigma$ and $\rho$
are conjugate, then 
   $w$ is a suffix or a prefix of $\rho(x)$   for $x =a,b$, and $\sigma(x)$  has the same length as the word 
   $\rho(x)$ for   every $x \in  \A$. 
\end{rem}

% \begin{example}
%Consider the two substitutions $\sigma: a \to ab, b \to a$  (see \eqref{eq:fib1}) and
%$\varrho: a \to ba, b \to a$.  We have $w=a$ as the conjugating word:
%\[ \gamma_a \varrho (a) = a \varrho (a) a^{-1} = abaa^{-1} = ab =
%\sigma(a) \] 
%and \[ \gamma_a \varrho(b) = a \varrho(b) a^-1 = aaa^{-1} = a =
%\sigma(b), \] 
%thus  for every $w=w_1\cdots w_n \in \A ^{\ast} $,  $\gamma_a \varrho(w) = a \varrho(w_1) a^{-1} a \varrho(w_2)
%a^{-1} \cdots a \varrho(w_n) a^{-1} = \sigma(w)$, hence $\varrho \sim
%\sigma$. We will see  in Section \ref{sec:sturm} that two conjugate substitutions generate
%the same hull. Therefore, both $\sigma$ and $\varrho$ produce
%the hull of  the Fibonacci sequences.
%\end{example}

% Let us note that we can extend the  notion of substitution  matrix  to 
% endomorphisms   of   the free  group  by taking  again the
% Abelianisation, which here    maps
% $F_2$ to $\Z^2$:  if $w$  is  a reduced word over $\{a,b,a^{-1}, b^{-1}\}$, $A(w)$ is the vector
% that counts the number of occurrences of each `positive' letter in $w$  minus 
% the number of occurrences of each `negative' letter. 
% For instance,  the substitution matrix of  $\sigma^{-1}: a \to b, b \to a^{-1}b$ is 
% $\big( \begin{smallmatrix} 0 & -1 \\ 
%  1& 1  \end{smallmatrix} \big)$.

Recall the following theorem of Nielsen \cite{Nielsen}: given two
automorphisms $\sigma$ and $\rho$ of the free group $F_2$, they
have the same   substitution matrix  if and only if they are
conjugate.  (Here, the substitution matrix relies   on the  Abelianisation map    
from  $F_2$ to $\Z^2$ 
 that counts the number of occurrences of each `positive' letter   minus 
 the number of occurrences of each `negative' letter.)
 We  thus conclude in   terms of  substitutions (see
\cite{Seebold98}  for a combinatorial  proof): 

\begin{thm}\label{thm:sturmconjug}  Two invertible substitutions
  $\sigma$ and $\rho$ are conjugate if and only 
  if they have the same  substitution  matrix. 
\end{thm}

\subsection{A rigidity result}  \label{sec:sturm}

The following theorem is a classical rigidity result
for two-letter  primitive substitutions   by P.~S\'e\'ebold
\cite{Seebold98}, see also  \cite{Krieger}. Here we provide two
alternative proofs. Let us note that the apparent simplicity  of the first 
proof  relies on the use  of  Theorem \ref{thm:sturmconjug} and the defect 
theorem (Theorem 1.2.5 in  \cite{Loth1}),  whereas the original proof  
of  \cite{Seebold98} was self-contained.

\begin{thm}  \label{thm:rigidity}
Let $\sigma, \rho$ be primitive substitutions on the two-letter  alphabet 
$\A$. If $\sigma^k \sim \rho^m$ for some $k, m$,
then $\Xc_{\sigma} = \Xc_{\rho}$.\\
Furthermore, if $\sigma$ and $\rho$ are invertible, then $\Xc_{\sigma}
= \Xc_{\rho}$  if and only if $\sigma^k \sim \rho^m$ for some $k, m$. 
\end{thm}

In plain words, this theorem states that two invertible substitutions are 
conjugate  (up to powers) if and only if their hulls are
equal. In even different words (compare Proposition \ref{uinxc}): if
$u$ is a bi-infinite word obtained by some invertible primitive substitution on two
letters, where the substitution is unknown, then $u$ determines the
substitution uniquely, up to  conjugation  and up to powers of the
inflation factor. An immediate consequence is the following result.

\begin{cor} \label{cor:xsxr}
Let $\sigma, \rho$ be primitive invertible substitutions on the
alphabet $\A = \{ a,b \}$ with the same  inflation factor. Then, 
$\sigma \sim \rho$  if  and only if $\Xc_{\sigma} = \Xc_{\rho}$, which is also equivalent to  $\sigma$ and $\rho$ having the same  frequency  $\alpha$. \hfill $\square$
\end{cor}

\begin{proof}  Let us prove Theorem \ref{thm:rigidity}.
Because of Proposition \ref{prop:ssn} we can restrict ourselves to the
case $k=m=1$. First we prove  that  if $\sigma \sim \rho$,  then
$\Xc_{\sigma}=\Xc_{\rho}$. Thus let $\sigma \sim \rho$. This
means there is $w \in  F_2$ such that $\sigma(x) = w \varrho(x) w^{-1}$ for $x
= a,b$. Consequently, for $x \in \A$ holds: 
$w \rho(x) =  \sigma(x) w$, and in general for  $u \in \A ^{\ast}$,  $w \rho(u) =  \sigma(u) w$, which yields
for  $x \in \A$,
\begin{equation} \label{eq:wrsw}
 w \varrho(w) \cdots \varrho^{k-1}(w) \rho^k(x) = 
\sigma^k(x) w   \rho(w) \cdots \rho^{k-1}(w). 
\end{equation} 
Setting $w^{(k)} : = w \rho(w) \cdots \rho^{k-1}(w)$, we can write shortly 
\[ w^{(k)} \varrho^k (x) = \sigma^k (x) w^{(k)}. \] 
Let $x=a$. By the defect theorem (Theorem 1.2.5 in  \cite{Loth1}) this 
commutation relation implies 
  \begin{equation} \label{eq:ukvk}
\rho^k(a) = u_k v_k, \quad \sigma^k(a) = v_k u_k 
\end{equation}
for some  words $u_k $ and $v_k$, with  $u_k$ being   nonempty. 
Now, let $u \in \Xc_{\sigma}$, and let $v$ be some subword  of  $u$. Then,
by primitivity, $v$ is a  subword of $\sigma^k(a)$ for 
some $k$. Since $\max(|u_k|,|v_k|) \to \infty$ for $k \to \infty$ ($\sigma$ is primitive), $v$ is also
contained in $u_k$ or $v_k$ for $k$ large enough. Thus it is also contained in
$\rho^k(a)$ for some $k$, thus $v$ is a word in
$\Xc_{\rho}$. The same argument holds vice versa, thus $\Xc_{\sigma} =
\Xc_{\rho}$. 

For the other direction, we first note that if $\lambda$ is
an irrational eigenvalue of a $2 \times 2$ integer matrix $M$ 
then it is an algebraic integer, and the algebraic conjugate
$\lambda'$ of $\lambda$ is the second eigenvalue of $M$. If
$v=(1,v_{\lambda})$ is an eigenvector  of $M$ corresponding to $\lambda$, then
$v'=(1,v'_{\lambda})$ (again, $v'_{\lambda}$ denotes the algebraic conjugate of $v_{\lambda}$)
is an eigenvector corresponding to $\lambda'$. 
If the matrix is    furthermore primitive,  both  vectors  are  distinct  by the 
Perron-Frobenius theorem, and thus linearly independent.
%Now we need the
%following lemma.  
%\begin{lem} \label{lem:conjug}
%Let $M_1, M_2$ be two primitive $2 \times 2$ matrices with the same 
%pair of real eigenvalues $\lambda > \lambda'$ and the same pair
%of eigenvectors  $v = (1, v_{\lambda})$ and $v' = (1, v'_{\lambda})$. Then $M_1 = M_2$. 
%\end{lem}
%\begin{proof}
%By assumption, there is $A \in GL(2,{\mathbb  R })$ such that $A^{-1} M_1A = 
%\big( \begin{smallmatrix} \lambda & 0 \\ 0 & \lambda' \end{smallmatrix}
%\big)$. 
%By elementary linear algebra, the columns of $A$ are eigenvectors
%of $M_1$. Therefore, $A$ can be written as 
% $A = r \tilde{A}$ with $\tilde{A}= \big( \begin{smallmatrix} 1 & s \\ v_{\lambda}
% & sv'_{\lambda}  \end{smallmatrix} \big)$, $r \ne 0 \ne s$.  We obtain $M_1=
%A \big( \begin{smallmatrix} \lambda & 0 \\ 0 & \lambda' \end{smallmatrix}
%\big) A^{-1} = \tilde{A} \big( \begin{smallmatrix} \lambda & 0 \\ 0 
%& \lambda' \end{smallmatrix} \big) \tilde{A}^{-1}$. A simple
%computation yields $M_1= \frac{1}{v_{\lambda} - v'_{\lambda}} \big( \begin{smallmatrix}
%  v_{\lambda} \lambda'-v'_{\lambda} \lambda & \lambda - \lambda' \\ v_{\lambda}
%  v'_{\lambda}(\lambda'-\lambda) & v_{\lambda} \lambda - v'_{\lambda}\lambda'
%\end{smallmatrix} \big)$, which is independent of $r$ and $s$.  

%Let $M_2$ be another matrix with the same eigenvalues and the same 
%eigenvectors. Then, the above holds also for $M_2$. It follows $M_1=M_2$.
%\end{proof}
%{\em Back  to the proof of  Theorem  \ref{thm:rigidity}.}
Let $\sigma$, $\rho$  be two   primitive and invertible substitutions such that $\Xc_{\sigma}=\Xc_{\rho}$.
Let $u \in \Xc_{\sigma}$.  Let us recall   that  the 
the  vector  of  frequencies  $(1-\alpha, \alpha)$  of letters in $u$
is an  eigenvector of the (up to here unknown)  
substitution matrices  $M_{\sigma}$ and $M_{\rho}$.  Hence the  vectors $ (1,\ell) =  
(1, \frac{\alpha}{1-\alpha})$ and $(1, \ell')$ are  eigenvectors for both matrices. 
We now  consider  the eigenvalues associated with  the  previous eigenvectors. Since $\sigma$ is
invertible,  $M_{\sigma}$ is unimodular. Its (up to here unknown)
inflation factor $\lambda$ is therefore a unit in the underlying
ring of integers   of the form $\Z[\sqrt{k}]$ for some $k \geq 1$. It is well known that the unit group of
$\Z[\sqrt{k}]$ is generated by a {\em fundamental unit} $z$. (This is
a consequence of the fact that there is a fundamental unit for the
solution of the corresponding Pell's equation, or a consequence of
Dirichlet's unit theorem, see for instance \cite{lang} or \cite{neuk}).   
Thus $\lambda$ is a power of the generating element $z$. Let 
$\lambda = z^n$, where $n$ is arbitrary but fixed.  The same  holds for  the   inflation  factor
of $M_{\rho}$ which  also belongs to   $\Z[\sqrt{k}]$,  and  which is thus of the form $z^m$. By algebraic
conjugation we obtain the second eigenvalue, and the second eigenvector  of $M_{\rho}$.

Since   the eigenvectors  $ (1,\ell) $ and $(1, \ell')$  are linearly  independent  eigenvectors,   on which 
the substitution  matrices    of  ${\sigma} ^n $
and ${\rho} ^m $  act in the same way, 
  ${\sigma} ^n $
and ${\rho} ^m $ have thus the same substitution  matrix.
  By the   fact
that all invertible substitutions with the same substitution matrix
are conjugate  (see Theorem  \ref{thm:sturmconjug}), the claim of the
theorem follows.   
\end{proof}

\begin{rem}
One can prove also directly from \eqref{eq:wrsw} that 
$\varrho^k(x)$ and $\sigma^k(x)$ share a common subword, without using 
the defect theorem. This can be done as follows:
We prove Equation \eqref{eq:ukvk} by the following 
argument:

We have to distinguish three cases (compare Figure
\ref{prefsuff}).

{\bf Case 1:} If $|w^{(k)}| = | \varrho^k (a)| =
|\sigma^k(a)|$, then $w^{(k)} = \varrho^k (a) = \sigma^k(a)$ (Figure
\ref{prefsuff}, left). 

{\bf Case 2:} If $|w^{(k)}| < | \varrho^k (a)| =
|\sigma^k(a)|$ (Figure \ref{prefsuff}, centre), then 
$w^{(k)}$ is a prefix of $\varrho^k(a)$ and a suffix of
$\sigma^k(a)$. Moreover, the remaining suffix of $\varrho^k(a)$
overlaps the remaining prefix of $\sigma^k(a)$. Thus we obtain 
again Equation \eqref{eq:ukvk}:
\begin{equation} 
\rho^k(a) = u_k v_k, \quad \sigma^k(a) = v_k u_k 
\end{equation}
for some nonempty words $u_k (= w^{(k)})$ and $v_k$. 

{\bf Case 3:}: If  
$|w^{(k)}| > | \varrho^k (a)| = |\sigma^k(a)|$ (Figure \ref{prefsuff},
right), then $\rho^k(a)$ is a suffix of $w^{(k)}$. Either this suffix
of $w^{(k)}$ overlaps already with $\sigma^k(a)$ (as in the figure),
then \eqref{eq:ukvk} holds for some nonempty words $u_k, v_k$. Or (if
$|w^{(k)}(a)| \ge 2 |\rho^k(a)|$)  $\rho^k(a) \rho^k(a)$ is a
suffix of $w^{(k)}$, and we continue with the shorter words $w^{(k)}$
vs $\sigma^k(a) w^{(k)}$ without the suffix $\rho^k(a)$. After
finitely many steps, we are in one of the first two cases.

\begin{figure}
\epsfig{file=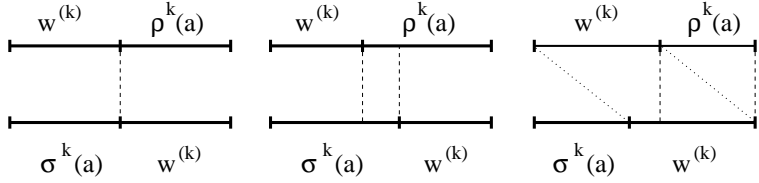,width=100mm}
\caption{Three cases in the proof of Theorem \ref{thm:rigidity}
\label{prefsuff}}
\end{figure}

In either case, \eqref{eq:ukvk} holds for some nonempty word
$u_k$. (In Case 1, just let $v_k$ be the empty word.)
\end{rem}

\begin{rem}
Let us note that  the   assumption  that $\sigma$ is invertible is
crucial in Theorem \ref{thm:rigidity}  as    shown   by the following
example (see   \cite{Krieger}). Consider    on the alphabet
$\{a,b\}$ the following two  primitive substitutions:
\[\sigma : \  a   \mapsto ab, \;   b \mapsto baabbaabbaabba, \
\rho : \ a  \mapsto   abbaab,  \; b \mapsto baabbaabba.\]

One   has $\sigma(ab)=\rho (ab)$ and $\sigma(ba)=\rho(ba)$.    
We deduce  that $\sigma$ and $\rho$ have the  same  fixed  point 
beginning   by $a$, and thus $\Xc_{\sigma}=\Xc_{\rho}$.   
Nevertheless, $\sigma$ and $\rho$  are
neither  conjugate, nor  conjugate up to a power  of
a common substitution (their substitution matrices are neither
conjugate in $GL(2,\Z)$, nor  conjugate up to a power  to a common matrix). 
\end{rem}

\section{Tile-substitutions} \label{sec:tilings}

\subsection{Substitution tilings}  \label{subsec:substil}
In contrast to word substitutions, which act on symbolic objects like
words, tile-substitutions act on geometric objects, like tiles or
tilings. A {\em tiling} of $\R^d$ is a collection of compact sets which
cover  topologically $\R^d$ in a non-overlapping way, that is, the interiors 
of the tiles are pairwise disjoint. In $\R^1$ there is a natural
correspondence between bi-infinite words and tilings when the tiles are  intervals: just assign to
each letter an interval of specified length.

In general, a tile-substitution in $\R^d$ is given by a set of 
{\em prototiles} $T^{}_1, \ldots, T^{}_m \subset \R^d$, an expanding map   
  and a rule how to dissect
each expanded prototile  into isometric copies of some
prototiles $T^{}_i$. Here, we restrict ourselves to two prototiles in
dimension one, and, moreover, our prototiles are always intervals,  with  the expanding map  being  given
 by an  {\em  inflation factor} $\lambda > 1$.
(For the discussion of analogues of some results of the present paper
for the case of more general tilings, see \cite{Rob,So,fre2}). The precise
definition of a tile-substitution in $\R$, where the prototiles are
intervals, goes as follows. 

\begin{defi} \label{defsubst}
A (self-similar) tile-substitution in $\R$ is defined via a set of
intervals $T_1, \ldots T_m$ --- the {\em prototiles} --- and a map
$s$. Let
\begin{equation} \label{eq:eifs}
 \lambda T^{}_j = \bigcup_{i=1}^m T^{}_i + \D^{}_{ij}  \quad (1 \le j \le m), 
\end{equation}
where the union is not overlapping (i.e., the interiors of the tiles
in the union are pairwise disjoint), and each $\D^{}_{ij}$ is a finite
(possibly empty) subset of $\R^d$, called {\em digit set}. Then  
\[ s( T^{}_j ) := \{ T^{}_i + \D^{}_{ij}  \mid  i=1 \ldots m
\} \]
is called a {\em tile-substitution}.  
It is called {\em primitive}  if the substitution matrix $M_s:= (|\D_{ij}|)_{1 \le i,j \le 2}$ is primitive, where 
$|\D_{ij}|$ stands for the   cardinality of  the set  $\D_{ij}$.
\end{defi}
By $s(T^{}_j+x):=s(T^{}_j)+\lambda x$ and $s(\{T,T'\})
:= \{ s(T), s(T') \}$, $s$ extends in a natural way to
all finite or infinite  sets of copies of the prototiles.

In analogy to word substitutions we want to deal with the space
$\X_{s}$ of all substitution tilings arising from a given
tile-substitution. Note the correspondence with the definition of the
hull of a word substitution, see Equation \eqref{eq:hull}. The main difference between both   types of associated  dynamical systems
is  that $\Xc_{s}$ is endowed with a $\mathbb{Z}$-action  by the shift,   whereas 
$\X_{s}$ is  endowed with an $\mathbb{R}$-action  defined by the  action of  translations.

\begin{defi} \label{tilhull}
Let $s$ be a primitive tile-substitution with prototiles $T_1,
T_2$. The {\em tiling space} $\X_{s}$ is the set of all tilings
$T$, such that each finite set of  tiles  of $T$ is contained in some translate
of $s^n(T_1)$ or $s^n(T_2)$.   Any element of  $\X_{s}$ is called  a  {\em substitution tiling}  generated by $s$.
\end{defi} 
\begin{rem} \label{digit}
Any primitive  self-similar tile-substitution is uniquely defined by its digit
set matrix $\D$. This holds because one can derive the inflation factor
$\lambda$ and the prototiles $T_i$ from the digit set matrix
$\D : = \big( \D_{ij} \big)_{ij}$. This is not only true for two tiles in one dimension, but for
any self-similar tile-substitution in $\R^d$. For details, see
\cite{fre2}. Here we just mention two facts: the inflation
factor $\lambda$ is the largest  eigenvalue of the primitive  substitution
matrix $M_s$. And the prototiles are  the
unique compact solution of the multi component IFS (iterated function
system)   in the sense  of \cite{MW} (also called graph-directed  IFS),  which is obtained by dividing \eqref{eq:eifs} by $\lambda$.   
% The matrix of this 
% IFS   is equal to $\D$.  
In particular
the vector of lengths of the tiles is  a  {\em  left } eigenvector of the substitution matrix $M_{\sigma}$.
\end{rem}
\begin{rem}\label{rem:wtot}
A one-dimensional tile-substitution (where the tiles are
intervals) yields a  unique word substitution: just replace the
tiles by symbols.
Conversely, one  can  realize  any   primitive word substitution  as  a tile-substitution  by taking as  lengths  $l_a, l_b$ for the prototiles  
 the coordinates of a  positive  left  eigenvector associated with its inflation  factor $\lambda$:  the intervals $T_i$ are chosen
  so that they line up with the action of the word substitution. We chose  here to normalise  the   eigenvector  $v_{\lambda}=(1,\ell_{\lambda})=(l_a,l_b)$ by taking its first coordinate equal to $1$. For $j=a,b$, let 
 $T_j=[0,l_j]$.
For $j=a,b$,  if $\sigma(j)=\sigma(j)_1\cdots  \sigma(j)_{|\sigma(j)|}$ then $\lambda l_j=\sum_{k=1}^ {|\sigma(j)|} l_{\sigma(j)_k}$, i.e.,
the tile $T_j$ is inflated by the factor $\lambda$, so it can be subdivided into translates of the prototiles according to the substitution rule:
\[
\lambda T_j=[0, \lambda l_j] \mapsto \{T_{\sigma(j)_1}, T_{\sigma(j)_2}+l_{\sigma(j)_1},\ldots, T_{\sigma(j)_{|\sigma(j)|}}+l_{\sigma(j)_1}+\cdots +l_{\sigma(j)_{|\sigma(j)|-1}}\}.
\]
This can be written as
\[
\lambda T_j= \bigcup_{i,j\,  : \, (j,k)\in F_i} T_j +\delta([\sigma(j)]_{k-1}),
\]
where 
$$F=\{(j,k) \mid j\in\A, \, 1\leq k\leq |\sigma(j)|\} $$ and 
$$F_i=\{(j,k)\in F \mid  \sigma(j)_k=i\},$$  with  the valuation map 
$\delta  \colon  \A ^{\ast} \rightarrow \mathbb{R}^+$ being defined   for any $w=w_1\cdots w_m \in \A ^{\ast}$ as 
$$\delta (w_1\ldots w_m)=l_{w_1}+\cdots +l_{w_m} = |w|_a  + |w|_b \ell_{\lambda}=\langle A(w), v_{\lambda} \rangle ,$$
by  recalling  that $A$ stands for the Abelianisation map.
 We then set for all $i,j$
 $${\mathcal D}_{ij}=\{\delta([\sigma(j)]_{k-1})\mid (j,k)\in F_i\}.$$

 \end{rem}
 We illustrate this  by the following example.

\begin{example}  \label{ex:fibsq}
Consider the square of the Fibonacci substitution from Equation
\eqref{eq:fib1}, namely, $\rho=\sigma^2: a \to aba, \, b \to ab$.  We work here with the square of the Fibonacci   substitution since the  determinant
of  its substitution matrix  equals
$1$. We 
will realize it as a tile-substitution as follows (see Figure
\ref{fig:fib}): Let $T_1 = [0,1], \, T_2 = [0,1/\tau]$, where $\tau =
\frac{\sqrt{5}+1}{2}$.   Note that   $\lambda=\tau ^2$ is   the  dominant eigenvalue for
$M_{\sigma^2}$, with    $(1, 1/\tau)$  being an  associated      left eigenvector.
Then  
\[ \tau^2 T_1 = [0,2 +1/\tau] = T_1 \; \cup \; (T_2 + 1) \; \cup \;( T_1 +
1+1/\tau); \quad \tau^2 T_2 = [0,1+1/ \tau ] = T_1 \; \cup \; (T_2 +1), \]   
where the unions are disjoint in measure. (Note that $\tau=1+1/\tau$.)  
Hence the last equation yields a tile-substitution $s$:
\[ s(T_1) = \{T_1, T_2  + 1, T_1 + \tau  \}, \quad
s(T_2) = \{ T_1 , T_2 +1 \}.  \] 
For an illustration of this substitution, see Figure \ref{fig:fib}.
This substitution $s$ can be encoded in the digit sets $\D_{1,1} = \{0, 
\tau \}, \D_{2,1} = \{1\}, \D_{1,2} = \{0\}, \D_{2,2} =
\{1 \}$. This can be written conveniently as a digit set matrix:
\begin{equation} \label{digitfib}
\D = \begin{pmatrix} \{0, \tau\} & \{0\} \\ 
\{ 1 \} & \{ 1 \} \end{pmatrix}.  
\end{equation}
By comparison with Definition \ref{def:prim} we note that we can derive
the substitution matrix from the digit set matrix simply as follows: 
$M_{\sigma^2} = \D= (|\D_{ij}|)_{1 \le i,j \le 2}$. In this case we
get the matrix $\big( \begin{smallmatrix} 2 & 1 \\ 1 &
  1 \end{smallmatrix} \big)$. Its dominant eigenvalue is the
inflation factor $\lambda = \tau^2$.
\end{example}

\begin{figure} 
\includegraphics[width=120mm]{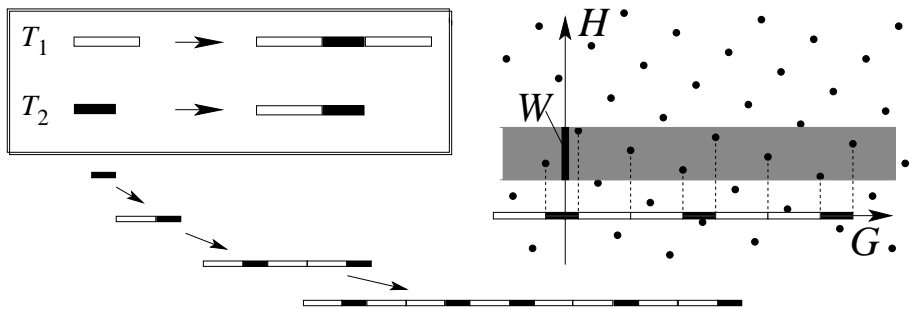}
\caption{The Fibonacci tile-substitution $s$ (box top left), some
  iterates of $s$ on $T_2$, and the generation of a Fibonacci
  tiling as a cut and project tiling (right). The interval $W$ defines
  a horizontal strip, all lattice points within this strip are
  projected down to the line.}
\label{fig:fib}
\end{figure}

\subsection{Cut and project tilings} \label{sec:cps}
In order to  define a notion of duality for  tile-substitutions,  we  now  work in the framework
of  cut and project  sets.

Certain substitution tilings can be obtained by a cut and project
method. There is a large number of results about such cut and project
sets, or model sets, see \cite{moo} and references therein.  In our
setting, this is pretty simple to explain, compare Figure
\ref{fig:fib}. Let $G=H=\R$, let $\pi_1 : G \times H \to G$, $\pi_2 :
G \times H \to H$ be the canonical projections, and let $\Lambda$ be a
lattice in $G \times H = \R^2$, such that $\pi_1: \Lambda \to G$ is
one-to-one, and $\pi_2(\Lambda)$ is dense in $H$. Then, choose some 
compact set $W \subset H$ with $W$  being the closure of its interior, and let 
\[ V = \{ \pi_1(x)  \mid  x \in \Lambda, \; \pi_2(x) \in W \}. \] 
Then $V$ is a {\em cut and project set} (or model set). Since $V$ is a
discrete point set in $\R = G$ ($W$  is   bounded),
 it induces a partition of $\R$ into
intervals. Regarding these (closed) intervals as tiles yields a tiling of
$\R$. Such a tiling is called {\em  cut and project tiling}.  Furthermore, a  tile-substitution $s$ is said to   yield  the   cut and project
set $V$  if   $V$ is the set of  left endpoints of  an  element of  $\X_s$.   According to Theorem  \ref{thm:intro},     a  primitive    word substitution $\sigma$ on two letters  
yields cut and project
tilings  whose window $W$  is an interval   if and only if it is  invertible \cite{lamb}.

Given 	 a   primitive  tile-substitution $s$ which is known   to   yield  a    cut and project set $V$,
one can construct $\Lambda$ and $W$ out of $s$ in a
standard way. In general, $\Lambda$ and $W$ are not unique. The
following construction provides a  choice of $\Lambda$ and $W$ when $s$  comes  from a 
primitive two-letter  substitution, according  to Remark \ref{rem:wtot}. It 
has the advantage that everything can be
expressed in some algebraic number field, which allows the use of
algebraic tools. 
 See  also   \cite{GMP03} for more details   on cut and project schemes
in  the
case of $G=H= \R$.  

We  thus 
start with a primitive   and unimodular word substitution $\sigma$ on a two-letter alphabet  such  that  its  associated    tile-substitution
$s$   yields  the    cut and project set $V$.
 Since  the substitution $M_{\sigma}$ is an integer matrix, its eigenvalues  $\lambda$  and $\lambda'$  are two conjugate
quadratic irrational numbers.   (The case where $\lambda$ is an integer requires
that the internal space $H$ is non-Euclidean \cite{bms,sing}. Since our
substitutions will always be unimodular in the sequel, this case
cannot occur here.) Let $v_{\lambda}=(1, \ell_{\lambda})$ be the  left  eigenvector associated  with the
dominant eigenvalue $\lambda$.  This eigenvector yields the  `natural'  lengths of the prototiles.
Thus let
$T_1 = [0,1]$, $T_2 = [0, \ell_{\lambda}]$. Note that   $1$ and $\ell$ are rationally  independent.
Now, let 
\begin{equation} \label{eq:lattice}
 \Lambda = \langle v, w \rangle_{\Z} = \{ \alpha  \big( \begin{smallmatrix} 1 \\ 1 \end{smallmatrix} \big)+ \beta \big( \begin{smallmatrix} \ell _{\lambda}\\ \ell' _{\lambda}\end{smallmatrix} \big)  \mid \alpha,\beta \in \Z
 \},
\end{equation}
where $\ell'_{\lambda}$ denotes the algebraic conjugate of $\ell_{\lambda}$. 
The  projections $\pi_1$ and $\pi_2$ correspond 
to the   canonical projections.
Now, consider the set $V$ of endpoints of the intervals in a
tiling in $\X_{s}$. Without loss of generality, let one endpoint be $0$. Then all other
endpoints are of the form $\alpha + \beta \ell_{\lambda} \in \Z[\ell_{\lambda}] \subset \Q(\lambda)$.  
Any point $\alpha + \beta \ell_{\lambda} \in V$ has a unique preimage in
$\Lambda$ (since $1$, $\ell_{\lambda}$ are  rationally independent), namely, $ \alpha  \big( \begin{smallmatrix} 1 \\ 1 \end{smallmatrix} \big)+ \beta  \big( \begin{smallmatrix} \ell_{\lambda} \\ \ell'_{\lambda} \end{smallmatrix} \big) $. Thus, each point $\alpha + \beta \ell_{\lambda}$ in $V$ has
a unique counterpart in the internal space $H$, namely $\pi_2 \circ
\pi_1^{-1} (\alpha + \beta \ell_{\lambda}) = \alpha + \beta \ell'_{\lambda}$. The map $\pi_2 \circ
\pi_1^{-1}$ is called {\em star map} and will be abbreviated
 by~$\star$. 

Since $\pi_2(\Lambda)$ is  dense in  $H$ and $W$ is the closure of its interior, the window $W$ is obtained as the closure of $\pi_2 \circ \pi_1^{-1}
(V)$. Note that the fact that $V$ {\em is} a cut and project tiling guarantees
that $W$ is  a   bounded   set,  since being  compact. With the help of the star map we can write
shortly $W = \cl(V^{\star})$. In our context, the star map has a very
simple interpretation: by construction of the lattice
$\Lambda$, the star map is just mapping an element of $\Q(\lambda)$ to
its algebraic conjugate
\begin{equation}\label{eq:star}
(\alpha + \beta \ell_{\lambda})^{\star} = \alpha + \beta \ell'_{\lambda}.
\end{equation}
Nevertheless, in general it is more complicated, and one should keep
in mind that the star map maps $G$ to $H$. In general, $G$ and $H$ can 
be very different from each other, for instance of different dimension. 

Now recall that any self-similar tile-substitution is uniquely defined
by its digit set matrix $\D$ (Remark \ref{digit}).  This allows us to   
define  the star-dual
 of a tile-substitution  by applying the star map to  the transpose $\D^T$ of the matrix  $\D$, as performed
 in \cite{thu,GE,fre2}  where this   `Galois'  duality for tile-substitutions is   developed.
%Let us recall    a  primitive    word substitution $\sigma$ on two letters  
%yields cut and project
%tilings  whose window $W$  is an interval   if and only if it is  invertible \cite{lamb}, see also  Theorem  \ref{thm:intro}. 

\begin{defi} \label{def:dualtilsubst}
Let $s$ be a  primitive  invertible self-similar tile-substitution yielding cut and project
tilings, with digit set matrix $\D$.
Then the {\em star-dual} substitution $s ^{\star}$  of $s$ is the unique 
tile-substitution defined by $(\D^T)^{\star}$,  with   the star map being  defined in  (\ref{eq:star}).
\end{defi}
Here $X^{\star}$ means the application of the star map to each element
of some set $X \subset \Q(\lambda)$ separately. This definition together with
Definition \ref{tilhull} defines the {\em star-dual} tiling space
$\X_{s^{\star}}$.    

\begin{example}\label{ex:DTstar}
The star-dual of the  tile-substitution associated with the squared Fibonacci substitution in Example
\ref{ex:fibsq} is easily obtained by applying the star map to the
transpose of the digit set matrix in \eqref{digitfib}. We obtain
\[ (\D^T)^\star = \begin{pmatrix} \{0, -\tau^{-1} \} & \{1  \} \\ \{ 0 \} & \{ 1 \} \end{pmatrix}. \]
\end{example}

\subsection{Natural decomposition method and Rauzy fractals} \label{sec:dcm}
We now explain  how to associate  a candidate window $W$ with any  primitive two-letter substitution $\sigma$  in   such a way that
 the  corresponding  tile-substitution $s$   yields a  cut and project set. 
  This candidate  is  the so-called Rauzy fractal associated with    the substitution $\sigma$ that is introduced below.

In this section we follow~\cite{sirw}, originally defined on a
$d$-letter alphabet. We restrict here to the case $d=2$.
As above, let $\sigma$ be a primitive unimodular substitution on the
two-letter alphabet $\A=\{a,b\}$. Let $\lambda>1$ be its inflation
factor and $\lambda'$ the other eigenvalue of the substitution  matrix
$M_{\sigma}$.  Let $u=(u_n)_{n \in \mathbb{N}}$ be an infinite word  such that $\sigma(u)=u$. (Note that we use one-sided infinite words here, in accordance 
with \cite{sirw}. The extension to bi-infinite words is straightforward.) 
Since $\sigma$ is primitive, it suffices to replace $\sigma$  by a suitable  power of $\sigma$ for such a  fixed point word to  exist.

The set $\Xc_{\sigma}$ is mapped into $\R$ via a valuation (compare with the valuation map $\delta$ introduced
in Remark \ref{rem:wtot}), which is a map
$\Delta:\A^*\rightarrow\R$, satisfying $\Delta(vw)=\Delta(v)+\Delta(w)$ and
$\Delta(\sigma(w))=\lambda' \Delta(w)$ for all $v,w\in\A^{*}$. Hence, $\Delta$
satisfies  for  all $w \in  \A^{*}$
\[
\Delta(w)=|w|_a +|w|_b  \ell_{ \lambda'}= 
\langle  A(w), v _{\lambda'}  \rangle,
\]
 where $v_{\lambda'}=(1,\ell_{\lambda'})$ is  the  normalised left eigenvector  associated  with  the eigenvalue $\lambda'$ of the
substitution matrix $M_{\sigma}$. 
According to \cite{Rauzy}, the set 
\[
\RR:=\overline{\{\Delta(u_0 u_1 \ldots u_m) \mid m\geq 0\}}
\]
is called the {\em Rauzy fractal} associated with the one-sided fixed
point $u$ of the substitution $\sigma$.  Since  $\sigma$ is unimodular, $|\lambda' | <1$ and one  deduces that  $\RR$ is a      compact set.
 For more about Rauzy fractals,  see  \cite{fogg}, \cite{BS05} or else \cite{BST} and the references therein.

Let
\[\RR_i:=\overline{\{\Delta(u_0 u_1 \ldots u_m)\mid u_{m+1}=i,\, \,
  m\geq 0\}}\] 
where $i\in\A$. Clearly $\RR=\RR_a\cup\RR_b$. We shall call
$(\RR_a,\RR_b)$ the {\em natural decomposition} of the Rauzy fractal
$\RR$. 
The natural decomposition of $\RR$ is the attractor of a graph
directed IFS (cf.~\cite{MW} and also \cite{HZ}) in the following way, with the notation of Remark \ref{rem:wtot}:
 $(\RR_a,\RR_b)$
satisfies 
\begin{equation}
\label{eqn1}
\RR_i=\bigcup_{(j,k)\in F_i} ( \lambda '\RR_j+\Delta([\sigma(j)]_{k-1}), \, \mbox{ for } \, i\in\A,
\end{equation}
with $F_i$ being defined in  Remark \ref{rem:wtot}.
To prove it, we  use the fact that $\sigma(u)=u$.

Note that the sets $\RR_i$ are the closure of their interior and their  boundary has zero measure, as 
proved in~\cite{sirw} in general case of a $d$-letter substitution.  For more properties, see  e.g.   \cite{fogg,BS05,BST}.
Furthermore,  the sets  $\RR_i$ are not necessarily intervals. In fact it can be proved that they are intervals if and only
if the two-letter primitive substitution $\sigma$  is invertible (see  also Theorem \ref{thm:intro}). For more details, see e.g. \cite{lamb,CAN,beir} for  proofs of this folklore result.

 A moment of thought yields that
$\Delta$ maps the $m$-th letter of $u$, which corresponds to the right
endpoint $|u|_a + |u|_b \ell_{\lambda}$ of the $m$-th tile in the tiling in $G$
(corresponding to the word $u$) to $|u|_a + |u|_b \ell'_{\lambda}$ in $H$. In
other words, $\RR$ is   a right candidate for the  window $W$ for the tiling in $G$, with
$V=\delta( \{u_0 \ldots u_m\mid   m\geq 0\})$ (see Remark \ref{rem:wtot} for the definition of the map $\delta$):
$$\RR= \overline{\pi_2 \circ \pi_1^{-1} \{\delta (u_0 \ldots u_m)\mid   m\geq 0\}},$$
with   $\pi_1$ and $\pi_2$ being  defined in Section \ref{sec:cps}.
In particular $ \pi_2 \circ \pi_1^{-1}  \delta (u_0 \ldots u_m)= \Delta   (u_0 \ldots u_m)$, for any $m \geq 0$.
It  remains to check that
$\{\pi_1(x) \mid  x \in \Lambda,  \  \pi_2(x) \in {\mathfrak R} \} = V$,  i.e., that
we do not have  $  V$ strictly included in $\{\pi_1(x) \mid  x \in \Lambda,  \  \pi_2(x) \in {\mathfrak R} \} $.
This comes from the following result.
\begin{thm}\cite{BD,Host,HS}
Let $\sigma$  be a primitive unimodular substitution on two letters.
Then $\sigma$ yields the  cut and project set  $V=\{\delta (u_0 \ldots u_m)\mid   m\geq 0\}$
with associated window $\RR$. 

\end{thm}

\begin{example}\label{ex:squarerauzy}
We consider the square of the Fibonacci substitution studied in Example~\ref{ex:fibsq}: 
$\rho: a\rightarrow aba$, $b\rightarrow ab$. 
Its    inflation factor  is $\lambda=\tau^2$, where $\tau=\frac{1+\sqrt{5}}{2}$. We define the valuation $\Delta$ with respect to the   left eigenvector $(1,\tau'-1)=
(1, \tau'^{-1})$
associated with the eigenvalue $\lambda'(=1/\tau^2).$

The natural decomposition $(\RR_a,\RR_b)$ of its Rauzy fractal $\RR$ is given by the solution of the equation
\[
\begin{array}{ccl}
\RR_a & = & \lambda' \RR_a+\Delta([\sigma(a)]_0) \cup (\lambda' \RR_a +\Delta([\sigma(a)]_2))\cup (\lambda' \RR_b+\Delta([\sigma(b)]_0))\\
\RR_b & = & (\lambda'\RR_a+\Delta([\sigma(a)]_1)) \cup (\lambda' \RR_b+\Delta([\sigma(b)]_1)).
\end{array}
\]
So the previous
equation can be written as 
\[
\begin{array}{ccl}
\RR_a & = & \lambda ' \RR_a \cup (\lambda' \RR_a +\tau')\cup  \lambda ' \RR_b= \lambda ' \RR_a \cup (\lambda' \RR_a -\tau^{-1})\cup  \lambda ' \RR_b\\
\RR_b & = & (\lambda' \RR_a+1) \cup ( \lambda' \RR_b+1) .
\end{array}
\]

The intervals $\RR_a=[-1,\tau^{-1}]$ and $\RR_b=[\tau^{-1},\tau]$
satisfy  this equation. \end{example}

Moreover, the natural decomposition of $\RR$ yields a tile-substitution
in $\R$, see Definition~\ref{defsubst}, in the following way: 
from (\ref{eqn1}) we get
\[
(\lambda')^{-1} \RR_i= \bigcup_{(j,k)\in F_i}
\big( \RR_j+\lambda'^{-1} \Delta([\sigma(j)]_{k-1}) \big), \, \mbox{ for } \, i\in\A.
\]
Note that we assume that $\det(M_{\sigma})=1$, thus $\lambda \lambda'
= 1$, which yields 
\begin{equation}\label{eqn:rauzy-tiling}
\lambda\RR_i =\bigcup_{j\in\A} (\RR_j + {\mathcal E}_{ji}),
\end{equation}
where
\begin{equation}\label{eq:dij}
{\mathcal E}_{ji}:=\{\lambda\Delta([\sigma(j)]_{k-1})\mid  (j,k)\in F_i\}.
\end{equation}
\begin{example}\label{ex:Efib}
We continue with  the square of the Fibonacci substitution studied in Example~\ref{ex:fibsq} and \ref{ex:squarerauzy}: 
$\rho: a\rightarrow aba$, $b\rightarrow ab$. 
The intervals $\RR_a=[-1,\tau^{-1}]$ and $\RR_b=[\tau^{-1},\tau]$
generate a dual tiling in $H=\R$ with the digit set 
\[
{\mathcal E}=\left(\begin{array}{cc}
             \{0,-\tau\} & \{\tau^2\} \\
             \{0\}  & \{\tau^2\}
             \end{array}
      \right).
\] 
\end{example}

The tiling obtained by~(\ref{eqn:rauzy-tiling})
is the dual tiling of the tiling described below, where duality is in
the sense of the cut and project scheme.  Indeed, we consider as in Remark \ref{rem:wtot}  the tile-substitution  $s$ associated with
$\sigma$    with  prototiles the intervals  $T_i$ of length $l_i$, for $i\in\A$,
by recalling that  the vector $v_{\lambda}=(l_a,l_b)$ is the positive left $\lambda$-eigenvector of $M_{\sigma}$ such that $l_a=1$.
Observe that we can associate  with the two-sided fixed point $u=\ldots 
u_{-1} .u_0u_1u_2\ldots$ of the substitution $\sigma$ the tiling
\[
\{\ldots T_{u_{-1}}-l_{u_{-1}}, T_{u_0}, T_{u_1}+l_{u_0}, T_{u_2}+l_{u_0}+l_{u_1},\ldots\}.
\]

According to Section \ref{sec:cps}, we  define the $\star$ map as:
\[\star:\Q(\lambda)\rightarrow \R, \, \, \lambda\mapsto \lambda'.\]
Since $\delta(vw)=\delta(v)+\delta(w)$ and 
$\delta(\sigma(w))=\lambda \delta(w)$, for any word $v,w\in \A^{*}$, 
it follows that $\Delta(w)=(\delta(w))^\star$, for all $w\in\A^*$.
Hence, by (\ref{eq:dij}),   one gets  $$({\mathcal D}_{ij}^{\star})_{ij}^T=\lambda^{-1}({\mathcal E}_{ij})_{ij}.$$

\begin{example}
One checks that  $({\mathcal D}_{ij}^{\star})_{ij}^T=\lambda^{-1}({\mathcal E}_{ij})_{ij}$ for the matrices of Example \ref{ex:DTstar}
and \ref{ex:Efib}.
\end{example}

\section{Dual maps of substitutions}  \label{sec:dual}
In this section, we present  a  notion of   substitution  whose  production rule
is a  formal  translation of   (\ref{eqn:rauzy-tiling}).

\subsection{Generalised  substitutions}  \label{subsec:E1}

We follow here the formalism introduced  in \cite{AI,sai}  defined
originally    on a  $d$-letter alphabet. We    restrict ourselves here
to the case $d=2$.  Let  ${\mathcal A}$ be the  finite  alphabet
$\{a,b\}$.

{\bf Finite  strand}
Let $(e_a,e_b)$ stand  for the canonical basis of ${\mathbb R} ^2$.
One  associates with each finite word $w=w_1w_2\ldots
w_n$ on  the two-letter alphabet ${\mathcal A}$  a path in the
two-dimensional space, starting from 0 and ending in $A(w)$,
with
vertices in $\{A(w_1\ldots w_i) \mid  
i=1\ldots n\}$: we start from 0, advance by $e_i$ if the
first letter
is $i$, and so on. For an illustration, see  Figure  \ref{fig:path} (left).

\begin{figure}
\begin{center}
{\begin{tikzpicture}
\draw[->](-1,0) -- (1.5,0);
\draw[very thick](0,0) -- (0.5,0);
\draw[very thick](0.5,0) -- (1,0);
\draw[very thick](1,0) -- (1,0.5);
\draw[very thick](1,0.5) -- (1.5,0.5);
\draw[->] (0,-1) -- (0,1.5); 
\end{tikzpicture}} 
\hspace{20mm}
{\begin{tikzpicture}
\draw[->](-1,0) -- (1.5,0);
\draw[very thick](0,0) -- (0,0.5);
\draw[very thick](-0.5,0) -- (0,0);
\draw[very thick](-1,0) -- (-0.5,0);
\draw[very thick](-1,0) -- (-1,-0.5);
\draw[->] (0,-1) -- (0,1.5); 
\draw(-1,-1) -- (1.5,1.5);\end{tikzpicture}}
\end{center}
\caption{The path associated with $aaba$ (left), an example of a  
finite strand with coding word $baab$ (right).}
\label{fig:path}
\end{figure}
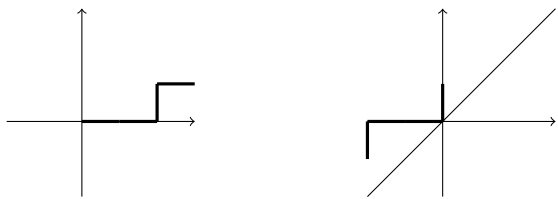

More generally, we define  the notion of  strand  by  following   the
formalism  of \cite{BK}. 
A  {\em finite  strand}   is a subset of  ${\mathbb R}^2$
defined as the image   by  a  piecewise isometric map
$\gamma \colon  [i,j] \rightarrow {\mathbb R} ^2$,  where  $i,j \in
{\mathbb Z}$, 
which  satisfies  the following: for    any integer $k \in [i,j)$,
there is  a letter  $x \in \{a,b\}$  such that 
$\gamma (k+1) -\gamma(k)=e_x$.  If  we replace
$[i,j]$  by ${\mathbb Z}$, we get  the notion   of    {\em  bi-infinite
  strand}.  
A  strand  is    thus a 
connected union of     unit  segments  with integer vertices  which
projects  orthogonally  
in a  one-to-one way
onto  the line $x=y$ (see  Figure \ref{fig:path}).  In particular, 
the path  associated  with a finite word $w$
such as  defined    in the  previous  paragraph  is a  finite  strand.  

We introduce the following notation for {\em elementary  strands}:
for   $W \in  \Z^2$ and  $i\in   {\mathcal A} $,  we set 
$ (W,i)= \{W+\lambda e_i\mid 0 \leq \lambda \leq 1\}.$  
%See  Figure \ref{fig:segmentsel} for an illustration.
%\begin{figure}
%\begin{center}
%{\begin{tikzpicture}
%\draw[->](-1,0) -- (1.5,0);
%\draw[very thick](0,0) -- (0.5,0);
%\draw[->] (0,-1) -- (0,1.5); \end{tikzpicture}}
%\quad \quad \quad \quad 
%{\begin{tikzpicture}
%\draw[->](-1,0) -- (1.5,0);
%\draw[very thick](0,0) -- (0,0.5);
%\draw[->] (0,-1) -- (0,1.5);
% \end{tikzpicture}
% }
% \end{center}
% \caption{The  elementary  strands $(0,a)$ (left) and   $(0,b)$ (right).
%}
% \label{fig:segmentsel}
%  \end{figure}

Any   bi-infinite strand  defines   a  bi-infinite   word  $w=(w_k)_{k
  \in {\mathbb Z}} \in \{a,b\}^{\mathbb  Z}$  
that satisfies $\gamma (k+1) -\gamma(k)=e_{w_k}.$
The  corresponding  map  which     sends  bi-infinite strands  on
bi-infinite words   is called  {\em     strand  coding}.

This
allows us to define a map on  strands, coming from the  word substitution,
by mapping
the strand  for $w$ to the
strand  for $\sigma(w)$.
In fact, this map can be  made into  a linear map, in the
following way.
Let $\sigma$  be a  substitution  on  ${\mathcal A}$. Let us recall the notation  for $i \in  \{a,b\}$:
\[\sigma(i)=\sigma(i)_1\cdots \sigma(i)_{|\sigma(i)|}=
\sigma(i)[k-1] \sigma(i)_k \sigma(i)_{k+1}\cdots \sigma(i)_{|\sigma(i)|} .\]

\begin{defi} \cite{AI,sai} We  let  ${\mathcal G}$ denote the real
vector space  generated by elementary strands. 
Let $E_1(\sigma)$ be the linear map defined on ${\mathcal G}$ by: 
\[E_1(\sigma)(W,i)=\sum_{k=1}^{|\sigma(i)|}(M_{\sigma}.W+A( \sigma(i)[k-1])
,\sigma(i)_k).\]
We call $E_1(\sigma)$ the  {\em one-dimensional extension} of $\sigma$. 
\end{defi}

\begin{defi} \label{def:Xb}
Let $\sigma$ be a primitive  substitution.
The {\em  strand  space}  $\Xb_{\sigma} $  is  the set   of bi-infinite
strands $ \eta $ such that each    finite substrand
$\xi $ of $\eta$ is   a substrand of some $E_1(\sigma)^n(W,x)$, 
for  $ W \in {\mathbb Z}^2$, $n \in {\mathbb N}$   
and $x \in \{a,b\}$.

\end{defi}

\subsection{Dual maps}\label{subsec:dualmaps}

 From now on, we suppose that $\sigma$ is a unimodular
substitution. In the sequel we will assume that $\sigma$ has determinant $+1$.
In view of Proposition \ref{prop:ssn} this is no restriction: if $\sigma$ has
determinant $-1$, we will consider $\sigma^2$ instead.

We want to study the dual map $E_1^*(\sigma)$ of $E_1(\sigma)$,
as a linear map on ${\mathcal G}$.
 We  thus denote by ${\mathcal G}^*$
  the dual space of  ${\mathcal G}$, i.e.,  the space of dual maps with
finite
support (that is, dual maps
that give value 0 to all but a finite number of the vectors of
the
canonical basis).

The space ${\mathcal G}^*$ has a natural basis $(W,i^*)$,  for $i=a,b$,  defined as  the map
that
gives value 1 to   
$(W,i)$ and 0 to all other elements of ${\mathcal G}$. It is possible to give
a
geometric meaning to this dual
space:   for $i=a,b$, we represent the element
$(W,i^{\ast})$ as  the  {\em lower}    unit segment  
perpendicular to the direction $ e_i$ of the unit   square with
lowest
vertex $W$.  By a slight abuse of notation,
$(W,i^{\ast})$   will stand both for the corresponding   dual map  and for  the   segment, i.e.,
$$(W,a^{\ast})=\{W +\lambda e_b\mid 0 \leq  \lambda \leq 1\}, \   (W,b^{\ast})=\{W +\lambda e_a\mid 0 \leq  \lambda \leq 1\}. $$
For an illustration, see  Figure \ref{fig:segmentsstar}.
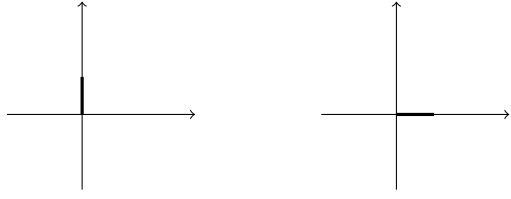
\begin{figure}
\begin{center}
{\begin{tikzpicture}
\draw[->](-1,0) -- (1.5,0);
\draw[very thick](0,0) -- (0,0.5);
\draw[->] (0,-1) -- (0,1.5); \end{tikzpicture}}
\quad \quad \quad \quad 
{\begin{tikzpicture}
\draw[->](-1,0) -- (1.5,0);
\draw[very thick](0,0) -- (0.5,0);
\draw[->] (0,-1) -- (0,1.5);
 \end{tikzpicture}
 }
 \end{center}
 \caption{The  segment $(0,a^{\ast})$ (left) and  the  segment $(0,b^{\ast})$ (right).}
 \label{fig:segmentsstar}
  \end{figure}
  Such a  segment is called  an {\em elementary dual strand}.

The map $E_1(\sigma)$ has a dual map,   which is easily computed:
\begin{thm} \cite{AI} \label{thm:defdual}
Let  $\sigma$ be a  unimodular substitution.  The dual map $E_1^*(\sigma)$
is defined
on ${\mathcal G}^*$ by
\[E_1^*(\sigma)(W,i^*)=\sum_{j,k :\  \sigma(j)[k]=i} \big(
M_{\sigma}^{-1}(W+A(\sigma(j)_{k+1}\cdots \sigma(j)_{|\sigma(j)|}), j^{\ast} \big).\]
Furthermore, if $\tau$ is  also a  unimodular substitution, then
\[E_1^{\ast} (\sigma \circ \tau)=E_1^{\ast}(\tau) \circ E_1^{\ast} (\sigma).\]
\end{thm}

\begin{example}\label{ex:generatorsbis}   We consider  the square of the Fibonacci substitution.
    Let  $  \rho \colon  a \mapsto   aba, b  \mapsto ab$.   One   has 
    $M_{\rho} =
\left(
\begin{array}{cc}
 2 &   1   \\
  1&   1  
\end{array}
\right)
  \mbox{  and }   M_{\rho}^{-1}=
\left(
\begin{array}{cc}
 1 &   -1   \\
  -1&   2   
\end{array}
\right).
$ We thus get 
$$\left\{\begin{array}{ll}E_1 ^* (\rho) (0,a^*)=(0,a ^{\ast})+(e_b, a^{\ast})+(-e_a+2e_b, b^{\ast})\\
   E_1 ^* (\rho)  (0,b^*)=(e_a-e_b,a^*)+(0,b^*).   \end{array}
   \right. $$
   \end{example}

To   be more  precise,  the definition of  the   map  $E_1^{\ast}(\sigma)$   in  \cite{AI}    involves prefixes instead of  suffixes,  
   whereas,
for $i=a,b$, the element
$(W,i^*)$  is represented by    the  {\em upper}  face
perpendicular to the direction $\vec e_i$ of the unit   square with
lowest
vertex $W$.  Nevertheless,   an  easy computation  shows that  both  formulas coincide.

{\bf Dual strands.}
We can  also   define  a notion of strand associated  with   this  dual  formalism.
 A  {\em finite   dual strand}   is a subset of  ${\mathbb R}^2$
defined as the image   by  a  piecewise isometric map
$\gamma \colon  [i,j] \rightarrow {\mathbb R} ^2$,  where  $i,j \in {\mathbb Z}$,
which  satisfies  the following: for    any integer $k \in [i,j)$,  there is  a letter  $x \in \{a,b\}$  such that
\[\gamma (k+1) -\gamma(k)=e_a \text{ if } x=a,  \    \gamma (k+1) -\gamma(k)=-e_b,\text{ otherwise}.\]
Segments $(W,x^*)$,  for $W \in {\mathbb Z}^2, x \in \{a,b\}$,  are    in particular    dual  strands.
   If  we replace
$[i,j]$  by ${\mathbb Z}$, we get  the notion   of    {\em  bi-infinite   dual strand}. 
A   dual strand  is     a 
connected union of     segments  with integer vertices  which
projects  orthogonally  
in a  one-to-one way
onto  the line $x+y=0$ (see  Figure \ref{fig:dualstrand}).
\begin{figure}
\begin{center}
{\begin{tikzpicture}
\draw[->](-1,0) -- (1.5,0);
\draw[very thick](-0.5,1) -- (0,1);
\draw[very thick](0,1) -- (0,0.5);
\draw[very thick](0,0.5) -- (0.5,0.5);
\draw[very thick](0.5,0.5) -- (0.5,0);
\draw[->][very thick](0.5,0) -- (0.5,-0.5);
\draw[->] (0,-1) -- (0,1.5); 
\draw (-1,1) -- (1.5,-1.5);\end{tikzpicture}} \end{center}
 \caption{An example of a  finite dual  strand coded  by the word $babaa$. The arrow at the end of the path  indicates that we read the letters from left to right.}
 \label{fig:dualstrand}
  \end{figure}
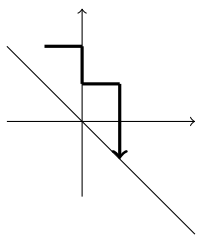

Any    bi-infinite  dual strand  defines   a  bi-infinite   word  $(w_k)_{k \in {\mathbb Z}} \in \{a,b\}^{\mathbb  Z}$ 
that satisfies  for  all $k$ 
$$\gamma (k+1) -\gamma(k)=e_a \mbox{ if } {w_k}=b, \mbox{ and } \gamma (k+1) -\gamma(k)=-e_b   \mbox{ otherwise}.$$  Similarly,
any   finite dual strand   $s$  defines  a    finite word $w$.
The map  $\psi^{\ast}$  that     sends   finite dual strands  on  words  in ${\mathcal A}^{\ast}$
is  called  {\em dual coding}. 
In particular   the word coding  $(W,a^*)$ is the letter  $b$,  and   the word coding $(W,b^*)$  is $ a$.

\subsection{Dual strands}
Before being able to  define the notion of dual strand space in Section   \ref{subsec:dualss},
we  need   to recall several facts on the  behaviour of  $E_1^{\ast} (\sigma)$ on  finite dual strands.

One   bi-infinite  dual strand plays  here a  particular role.
Recall that  $v_{\lambda}$  stands for    a  positive   left eigenvector   of  the substitution  matrix  $M_{\sigma}$ of   the primitive   substitution  $\sigma$ 
associated  with the    inflation factor $\lambda$.
Let  $\alpha$  be  the    frequency of $\sigma$.
We define ${\mathcal S}_{\alpha}$  as  the   union  of   segments $(W,i^*)$, for $i=a,b$, that satisfy
 $$0 \leq \langle W,v_{\lambda}\rangle  < \langle e_i, v_{\lambda} \rangle.$$
One checks that   ${\mathcal S}_{\alpha}$ is a  bi-infinite  dual strand. For more details, see e.g.  \cite{beir}.
One key property  is that    this  bi-infinite  dual    strand is preserved under the action of $E_1  ^{\ast}(\sigma)$.

\begin{thm}\cite{AI}\label{thm:stable}
Let $\sigma$ be a  unimodular   primitive  two-letter substitution with frequency $\alpha$. The 
map $E_1^{\ast} (\sigma)$  maps  any elementary dual   strand  of ${\mathcal S}_{\alpha}$
 on a     finite union of    elementary dual   strands    of ${\mathcal S}_{\alpha}$.
 Furthermore,   if $(V,i^{\ast})$ and $(W,j^{\ast})$  are two  distinct    segments  
 included in  some ${\mathcal S}_{\alpha}$,  for  $\alpha  \in (0,1)$,
then the intersection of   their  images  by $E_1^*(\sigma)$    is either  empty, or reduced to a point.

\end{thm}

Note that if $\sigma$ is not invertible,  the  image  by  $E_1^{\ast} (\sigma)$
of a  finite dual strand might not be connected.  However, if $\sigma$ is invertible, 
connectedness is preserved:
 finite dual    substrands  of ${\mathcal S}_{\alpha}$ (i.e.,  connected  unions of  segments)  are preserved under 
 $E_1^{\ast} (\sigma)$. Note that  a  proof  different   from the following one can be found in \cite{Ei-Ito}.

\begin{prop}\label{prop:dualstrand}

Let $\sigma$ be     a   primitive invertible  two-letter   substitution.
 The map $E_1 ^{\ast} (\sigma)$  maps    every finite strand   onto  a finite strand.

\end{prop}

\begin{proof}We first   check   that  
 the generators
$E,   L,  \tilde{L}$ of the monoid of invertible two-letter  substitutions     (see (\ref{eq:generators})) map   every finite dual   strand made  of two   adjacent segments
to some  finite dual  strand:

$\left \{\begin{array}{ll}
&E_1^* (L)(  (x,a^{\ast})+ (x,b^{\ast}))=(x,a^{\ast})+(x,b^{\ast})+(x-e_a+e_b,b^{\ast})\\
&E_1^* (L)(  (x,a^{\ast})+ (x+e_b,a^{\ast}))=(x,a^{\ast})+(x-e_a+e_b,b^{\ast})+
(x-e_a+e_b,a^{\ast})+(x-2e_a+2e_b,b^{\ast})\\
&E_1^* (L)(  (x,b^{\ast})+ (x+e_a,b^{\ast}))=(x,b^{\ast})+(x+e_a, b^{\ast})\\
&E_1^* (L)(  (x,b^{\ast})+ (x+e_a-e_b,a^{\ast}))=(x,b^{\ast})+(x+e_a,b^{\ast})+(2e_a-e_b,a^{\ast})
\end{array}
\right.$

$\left \{\begin{array}{ll}
&E_1^* (\tilde{L})(  (x,a^{\ast})+ (x,b^{\ast}))=(x,a^{\ast})+(x,b^{\ast})+(x+e_a,b^{\ast})\\
&E_1^* (\tilde{L})(  (x,a^{\ast})+ (x+e_b,a^{\ast}))=(x,a^{\ast})+(x,b^{\ast})+(x-e_a+e_b,a^{\ast})+(x-e_a+e_b,b^{\ast})\\
&E_1^* (\tilde{L})(  (x,b^{\ast})+ (x+e_a,b^{\ast}))=(x+e_a,b^{\ast})+(x+2e_a, b^{\ast})\\
&E_1^* (\tilde{L})(  (x,b^{\ast})+ (x+e_a-e_b,a^{\ast}))= (x+e_a,b^{\ast})+(x+2e_a-e_b,a^{\ast})+
(2e_a-e_b,b^{\ast})
\end{array}
\right.$

$\left \{\begin{array}{ll}
&E_1^* (E)(  (x,a^{\ast})+ (x,b^{\ast}))=(x,a^{\ast})+(x,b^{\ast})\\
&E_1^* (E)(  (x,a^{\ast})+ (x+e_b,a^{\ast}))=(x,b^{\ast})+(x+e_a,b^{\ast})\\
&E_1^* (E)(  (x,b^{\ast})+ (x+e_a,b^{\ast}))=(x,a^{\ast})+(x+e_b, a^{\ast})\\
&E_1^* (E)(  (x,b^{\ast})+ (x+e_a-e_b,a^{\ast}))= (x,a^{\ast})+(x-e_a+e_b,b^{\ast}).
\end{array}
\right.$

Hence     these generators   map   finite dual  strands  to
connected unions  of   unit  segments with integer vertices.  It
remains  to  check  that these    unions    are indeed    dual
strands (i.e.,  that they can be projected  orthogonally   in  a one-to-one way    to  $x+y=0$). By  Theorem \ref{thm:stable}, 
they  all   are substrands of ${\mathcal S}_{\alpha}$. We deduce  that
the generators   map   finite dual  strands  to    finite  dual
strands.

Let us prove now  that invertible substitutions  map    finite dual  strands  to    finite  dual
strands.  Let  $\tau$  be a  two-letter  substitution that maps   every  finite
dual  strand  to    a  finite  dual  one.  Now, if $\sigma= \tau
\circ L $,   then we deduce from    $E_1^{\ast} (\sigma)=E_1^{\ast}  (L)
\circ  E_1^{\ast} (\tau)$, that      the map $E_1 ^{\ast} (\sigma)$
also maps    every finite dual  strand to  a     connected union  of
unit  segments,  and hence by  Theorem \ref{thm:stable},   to  a
finite  dual strand.     The  same  holds  true  for the other
generators.  We thus   conclude  by induction     on the  length  of  a
decomposition on the  generators $E,L, \tilde{L}$. \end{proof}

\subsection{Dual strand space and dual substitution}\label{subsec:dualss}
  
We now can  introduce the notion  of  dual  strand space for  an invertible substitution.
\begin{defi} \label{def:Xbdual}
Let $\sigma$ be a primitive invertible   substitution over a
two-letter  alphabet. The {\em   dual strand  space}
$\Xb_{\sigma}^{\ast} $  is  the set   of bi-infinite  dual  strands $ \eta $ 
such that each    finite substrand
$\xi $ of $\eta$ is   a substrand of some $E_1 ^{\ast}(\sigma)^n(W,x^{\ast})$, 
for  $ W \in {\mathbb Z}^2$, $n \in {\mathbb N}$   
and $x \in \{a,b\}$.

\end{defi}
According to Proposition   \ref{prop:dualstrand}, since  the image 
by  $E_1 ^{\ast}(\sigma)$ of a   finite strand  is   a finite strand,   it can  be coded    as a  substitution
via the  dual coding   $\psi ^{\ast} $ (introduced in Section \ref{subsec:dualmaps}).

\begin{defi}\label{defi:dualsubstitution}
Let $\sigma$ be a primitive invertible  word   substitution over $\{a,b\}$ whose  substitution matrix has determinant $1$.
 The {\em dual substitution}     $\sigma^{\ast}$
is  defined on the alphabet $\{a,b\}$ as  
\[ \sigma^{\ast} (x)=\psi ^{\ast} (E_1^{\ast}(\sigma) (0,x^{\ast}
))\text{ for  }x  =a,b.\] 
\end{defi}

One has   $L^{\ast} \colon a\mapsto ba,\, b \mapsto  b,$
${\tilde L}^{\ast} \colon a\mapsto ab,\, b \mapsto  b$ and
$E^{\ast} \colon a\mapsto b,\, b \mapsto  a.  $
The   substitution    $L^{\ast} $ is   usually denoted as 
$R$. (we will use it in  Appendix  \ref{sec:arithm}).

\begin{example}  \label{ex:dual2}
Let $\rho\colon a \mapsto aba, \ b \mapsto ab$    be the square of the Fibonacci substitution.
One has $\rho^{\ast}\colon  a \mapsto baa, \ b \mapsto ba.$
\end{example}
\begin{rem}\label{rem:sigmastar}
Let us observe that the substitution matrix of $\sigma^*$   is the
transpose of the substitution matrix of $\sigma$.  Indeed,   the dual strand    coded  by $\sigma^{\ast}(0,a^{\ast})$ is located
on the left  of   the dual strand    coded  by $\sigma^{\ast}(0,b^{\ast})$ (we use the fact that $M_{\sigma}$ has determinant $+1$). Furthermore, one  checks that   $\sigma^*$
is invertible (it suffices to check it   on the generators $E, L, \tilde{L}$).
%This can be deduced for instance from Theorem \ref{thm:intro} 
%together with Example  \ref{ex:dual2}
Furthermore,
the   inflation factor   of $\sigma^{\ast}$  is
equal  to  the  inflation factor  $\lambda$   of $\sigma$.
 Furthermore, $\sigma$ and $\tau$ are conjugate if
 and only if $\sigma^{\ast}$ and $\tau^{\ast}$ are conjugate.
\end{rem}
This allows us to  define a   notion of dual    frequency.

\begin{defi}
The dual frequency $\alpha^{\ast}$ of $\sigma$ is defined as 
the frequency of $\sigma ^{\ast}$. 
\end{defi}
The dual  frequency is closely related to the algebraic conjugate of the  frequency. Indeed, one has the following result.
%whose proof is  given Appendix \ref{sec:arithm}.

\begin{thm}  \label{thm:conjugue}
Let  $\sigma$ be a  primitive  two-letter   invertible  substitution whose  substitution matrix has determinant $1$. Let
$\alpha$  be its      frequency and $\alpha'$ its algebraic conjugate.
One has 
$$\alpha^{\ast}= \frac{\alpha' -1}{2\alpha'-1}$$
\end{thm}

\begin{proof}
According to Remark \ref{rem:sigmastar},  the substitution matrix  of  $\sigma ^{\ast}$  is the
transpose of the substitution matrix of $\sigma$. The   eigenvector of  frequencies 
$(\alpha ^{\ast}, 1- \alpha ^{\ast})$  of 
 $\sigma ^{\ast}$ is thus orthogonal to the  eigenvector  
 $(\alpha' , 1- \alpha' )$ of  $M_{\sigma}$. Indeed, they have distinct eigenvalues  which are respectively
 $\lambda$ and $\lambda'$.
 This yields $\alpha ^{\ast} \alpha '+ (1- \alpha ^{\ast}) (1-\alpha')=0$, which gives the  desired  conclusion. \end{proof}

\begin{rem}
There  exist  several codings for  dual strands and dual substitutions  that are possible  \cite{ei,beir,BR08}.
One could indeed  code strands from right to left,  or  exchange the  roles played by $a$ and $b$.
In particular,   the  coding  developed by \cite{ei}     yields  $\widetilde{\sigma^{-1}}$,
where  $\tilde{\tau}$ is  the substitution deduced    from $\tau$ by reading the letters  in   the images  of  letters  by  $\tau$    in reverse order.
The substitution  matrices of  these substitutions   are either  equal  or    transpose  of the  substitution matrix of  $\sigma$.
Furthermore, their  frequency  belongs to $\{\alpha^{\ast}, 1- \alpha^{\ast}\}$.
Note also  that  a  similar   adequate coding  can be introduced if the  substitution matrix $M_{\sigma}$  has determinant $-1$.
\end{rem}

\section{Relations between distinct concepts of `dual substitution'} \label{sec:relations}

In this section we will compare the various notions of duality we have
introduced so far.  For this purpose we will introduce the appropriate
equivalence relations among  hulls,  tiling spaces etc. But first we
will have a closer look on the inverse of a word substitution.

 \subsection{Equivalences}  \label{subsec:conjug}

If $\sigma$ is a primitive invertible substitution on $\{a,b\}$, then it is
clear that $\sigma^{-1}$ cannot  be a substitution on the 
original letters $a,b$, since for the Abelianisation would follow:
$M_{\sigma} M_{\sigma^{-1}} = M_{\sigma} (M_{\sigma})^{-1} = \id$, but
this is impossible for primitive $M_{\sigma} \ge 0$ and
$(M_{\sigma})^{-1} \ge 0$. But we have the following:

\begin{prop}
Let $\sigma$ be an invertible substitution over $\{a,b\}$. Assume that  its substitution matrix
$M_{\sigma}$ has  determinant $+1$.  Then its
inverse is a substitution on the two-letter alphabet $\{a^{-1}, b\}$.
Furthermore if $\sigma$ is  primitive, then  its inverse is also  primitive.
\end{prop}

This is a point where we need $\det(M_{\sigma}) = 1$. For
$\det(M_{\sigma}) = -1$, things are more difficult ($\sigma^{-1}$
would mix $a^{-1}, b$ with $a, b^{-1}$). In such a case we consider
$\sigma^2$.  

\begin{proof} 
That $\sigma^{-1}$ is indeed a substitution over $\{a^{-1}, b\}$ with
Abelianisation  $M_{\sigma^{-1}}$ can easily be checked by an induction
on the length of a decomposition of $\sigma$ on the set of generators
$\{E,L,\tilde{L}\}$ (see (\ref{eq:generators})). It suffices to consider $L^{-1}$ and
$\tilde{L}^{-1}$, they are substitutions over $\{a^{-1},b\}$.
\end{proof}

In Section \ref{sec:selfdual} we will address the question in which
cases the substitution $\sigma^{-1}$ yields the same hull as $\sigma$ 
(i.e., the same  set of  bi-infinite words),
up to renaming letters. We know already that $\sigma^{-1}$ is always a 
substitution on $\{a^{-1}, b\}$. This leaves two possibilities for
renaming the letters, namely $a \mapsto a^{-1}, \, b \mapsto b$, or 
$a \mapsto b^{-1}, \, b \mapsto a$. Both cases will occur: see Example  \ref{ex:fibsq2} and   \ref{ex:unchange} below where 
  two examples  of substitutions are given   whose   respective inverses yield the same  hull as the original substitutions,
up to renaming letters.  Thus we
define the `inverse' substitution (that we call reciprocal) as $\tau_a \sigma^{-1} \tau_a$, where
$\tau_a$ exchanges $a$ with $a^{-1}$, and later (see Definition \ref{def:selfdual}) we will regard $\sigma$
as selfdual, if its  reciprocal is either conjugate to  $\sigma$ itself, or  to $E \sigma
E$, with $E: a \to b, \; b \to a$.
This notion       coincides with the  notion of  duality introduced in \cite{BR08}, see  also in the same flavour \cite{KR}.

\begin{defi} \label{def:invsubs}
Let $\sigma$ be a substitution over $\A = \{a,b\}$ with determinant $1$, and let $\tau_{a} \in
\mbox{Aut}(F_2)$, $\tau_a: a \to a^{-1}, \; b \to b$. If $\sigma$ is
invertible, then we call $\overline{\sigma} := \tau_a
\sigma^{-1}\tau_a^{-1}$ the {\em   reciprocal substitution} of $\sigma$. 
\end{defi}
\begin{rem}\label{rem:factor}
If $\det M_{\sigma}= 1$, then 
$M_EM_{\overline{\sigma}}M_E = M_{\sigma } ^T.$
Let $M=\big( \begin{smallmatrix} p & q \\ r & s \end{smallmatrix}
\big)$. Then, since $\det M_{\sigma} =1$, one has
$M_{\sigma^{-1}} 
 = \big( \begin{smallmatrix}
     s & -q \\ 
     -r & p \\ 
   \end{smallmatrix}\big),$
   and  the result follows from
 $M_{\overline{\sigma}} = \big( \begin{smallmatrix} -1 & 0
  \\ 0 & 1 \end{smallmatrix} \big)  M_{\sigma} ^{-1}
\big( \begin{smallmatrix} -1 & 0 \\ 0 & 1 \end{smallmatrix} \big).$
In particular, the substitutions $\sigma$ and $\overline{\sigma} $   have the same inflation factor.
\end{rem}

\begin{example} \label{ex:fibsq2}
The square of the Fibonacci substitution (see Example \ref{ex:fibsq})
is the substitution $\rho\colon  a \to aba, \, b \to ab$. 
Then, $\rho^{-1} \colon  a^{-1} \to  a^{-1}b, \, b \to a^{-1}bb$, 
and 
\[ \overline{\rho} (a) = \tau_a \rho^{-1} \tau_a^{-1} (a) = \tau_a
\rho^{-1} (a^{-1}) = \tau _a(a^{-1}b) = ab. \]
\[ \overline{\rho} (b) = \tau_a \rho^{-1} \tau_a^{-1} (b) = \tau_a
\rho^{-1} (b) = \tau_a ( a^{-1} b b) = abb. \]
Thus the reciprocal substitution of $\rho$ is $\overline{\rho} : a \to
ab, \; b \to abb$. 

Now we can see that $\rho$ and $\overline{\rho}$ are conjugate, after
renaming the letters  according to $E: a
\to b, \; b \to a$. Indeed, the word $w=a$ yields a conjugation between $\rho$ and
$E \overline{\rho} E$. 
\begin{align*}
a E \overline{\rho} E (a) a^{-1} & = a baa a^{-1} =
aba = \rho(a),\\
a E \overline{\rho} E (b) a^{-1} & = a ba a^{-1} = ab = \rho(b).
\end{align*}
In particular, it follows that the hulls $\Xc_{\rho}$ and
$\Xc_{\overline{\rho}}$ are equal (up to renaming letters). This is a
case where the letters change their role: in $\Xc_{\rho}$ the letter
$a$ is more frequent, whereas in $\Xc_{\overline{\rho}}$ the letter $b$  is 
more frequent.\end{example}

\begin{example}\label{ex:unchange}
Consider $\sigma: a \to abaab, \; b \to ababaab$. Then, $\sigma^{-1} \colon
a^{-1} \to b a^{-1} a^{-1}b a^{-1}, \; b \to b  a^{-1} a^{-1} b a^{-1}
b a^{-1}$. Thus $\overline{\sigma}\colon  a \to baaba, \; b \to
baababa$. Here, $w=b^{-1}a^{-1}a^{-1}b^{-1}$ yields a conjugation
between $\sigma$ and 
$\overline{\sigma}$:
\begin{align*}
w \overline{\sigma}(a) w^{-1} & = b^{-1}a^{-1}a^{-1}b^{-1} baaba baab
= abaab = \sigma(a)\\
w \overline{\sigma}(b) w^{-1} & = b^{-1}a^{-1}a^{-1}b^{-1} baababa
baab = ababaab = \sigma(b)
\end{align*}
This is an example where the letters do not change their role.
\end{example}

%\begin{example} \label{ex:generators}
%Consider $L \colon a \to a, \; b \to ab$. One has 
%$L^{-1} \colon
%a^{-1} \to  a^{-1} , \; b \to   a^{-1} b
%$. Thus $\overline{L} \colon  a \to a, \; b \to ab$. One has $\overline{L}=L$.

%Consider $R \colon a \to ba, \; b \to b$. One has 
%$R^{-1} \colon
%a^{-1} \to  a^{-1}b , \; b \to    b
%$. Thus $\overline{R} \colon  a \to ab, \; b \to b$. 
%\end{example}

Now we want to compare hulls with tiling spaces and strand spaces.
Thus we introduce suitable  equivalences of these spaces. 

\begin{defi}
Two hulls $\Xc_{\sigma}, \Xc_{\rho}$ over $\{a,b\}$ are 
{\em equivalent}, short $\Xc_{\sigma} \cong \Xc_{\rho}$, if there is a
letter-to-letter  morphism $\tau$, such that $\Xc_{\sigma} =
\tau(\Xc_{\rho} ) = \{ \tau(u) \, | \, u \in \Xc_{\rho} \}$.  \end{defi}

 In other words,   $\Xc_{\sigma}$ and $ \Xc_{\rho}$ are 
 equivalent either if $\Xc_{\sigma} = \Xc_{\rho}$, or  if $\Xc_{\sigma} = E(\Xc_{\rho})$,
where $E: a \to b, \; b \to a$. 

\begin{defi}
Two tilings $\T, \T'$ are called {\em equivalent}, short: $\T \cong
\T'$, if they are {\em similar}, i.e., there are $c>0, t \in \R$ such
that $c \T + t = \T'$. Such a map is called {\em similarity}. Let
$s,s'$ be primitive tile-substitutions. The tiling spaces $\X_s$ and
$\X_t$ are called equivalent, short: $\X_s \cong \X_{s'}$, if there is
a one-to-one similarity mapping $\X_s$   to $\X_{s'}$. 
\end{defi}

By  the analogue of Prop.\ \ref{uinxc}  for tiling space,
$\X_s \cong \X_{s'}$  whenever there are $\T \in \X_s, \T' \in \X_{s'}$ such
that $\T \cong \T'$, with $s,s'$ being primitive.

 We have seen in Section \ref{subsec:substil} that a   word substitution $\sigma$ yields a one-dimensional
  tile-substitution in a canonical way: the substitution matrix $M_{\sigma}$
yields the inflation factor $\lambda$, and the     left eigenvector of
$\lambda$ (unique up to scaling) yields the tile lengths. This, together with
the order of the tiles in $\sigma$, yields the digit set matrix $\D$.
Vice versa, a one-dimensional tile-substitution (where the tiles are
intervals) yields a  unique word substitution: just replace the
tiles by symbols. Whenever a word substitution $\sigma$ is linked
with a tile-substitution in this manner, we say $\sigma \cong s$. If
one of them (thus both) are primitive, then we write $\Xc_{\sigma}
 \cong \X_s$ whenever  $\sigma \cong s$.

Using the coding  maps (strand  coding, dual strand  coding), we can
define equivalence between strand spaces and dual strand spaces as
follows.
\begin{defi}   Let $\sigma, \sigma'$ be two unimodular  primitive substitutions over  a two-letter alphabet.
We define $\Xc_{\sigma} \cong \Xb_{\sigma'}$ if  $\psi
(\Xb_{\sigma'}) = \Xc_{\sigma}$, and  $\Xc_{\sigma} \cong \Xb^{\ast}_{\sigma'}$ 
if $\psi ^{\ast}(\Xb_{\sigma'}) = \Xc_{\sigma}$.
Furthermore, one  has  $\Xb_{\sigma} \cong \Xb_{\sigma'}$ if  $\psi
(\Xb_{\sigma'}) = \psi( \Xb_{\sigma})$.  
\end{defi}
Note that
$$\Xb^{\ast}_{\sigma} \cong \Xb_{\sigma^{\ast}} \cong   \Xc_{\sigma^{\ast}} .$$
\subsection{Equivalence theorem}
Let us summarise these considerations.

\begin{rem} \label{thm:equiv}
Let $\sigma$ be an invertible substitution on two letters, and let $s$
be some tile-substitution  derived from $\sigma$ as described
above. Then, by definition, \[ \Xb_{\sigma} \cong \Xc_{\sigma} \cong \X_{s}.\]  
\end{rem}

Now we can state the main theorem of this section:

\begin{thm}\label{thm:inv}
Let $\sigma$ be a  primitive  invertible substitution on two letters with determinant $1$.
Then 
\[ \Xb^{\ast}_{\sigma} \cong \Xb_{\sigma^{\ast}} \cong
\Xc_{\overline{\sigma}} \cong \Xc_{\sigma^{\ast}} \cong
\X_{s^{\star}}.\]  
\end{thm}

A different  proof  of the equivalence  between   $\sigma^{\ast}$ and
$E^{\ast}_1(\sigma)$ can also be found   in \cite{BR08}.

\begin{proof} 
In order to show the equivalence of $s^{\star}$ and
$E^{\ast}_1(\sigma)$, note that it does not matter on which lattice 
the paths in $E_1(\sigma)$ are defined: all lattices of full rank in
$\R^2$ are isomorphic.  We use here the notation of Section \ref{sec:cps}. So, let the underlying lattice be $\Lambda =
\langle \big( \begin{smallmatrix} 1 \\ 1 \end{smallmatrix} \big),   
        \big( \begin{smallmatrix} \ell_{\lambda} \\ \ell_{\lambda}^{\star}
\end{smallmatrix} \big) \rangle_{\Z}$, rather than $\Z^2$. 
The substitution matrix
$M_{\sigma}$ acts as an automorphism on $\Lambda$.
 Indeed, since the projection $\pi_1$ from
$\Lambda$ to the first coordinate is $(1, \ell_{\lambda})
\big( \begin{smallmatrix} \alpha \\   \beta \end{smallmatrix}
\big) = \alpha + \beta \ell_{}$, and 
$(1,\ell_{\lambda})$ is a left  eigenvector of $M_{\sigma}$, multiplication by
$M_{\sigma}$ in $\Lambda$ acts as multiplication by $\lambda$ in
$\Z[\ell_{\lambda}]$: 
\[ (1,\ell_{\lambda}) M_{\sigma} \big( \begin{smallmatrix} \alpha \\
  \beta \end{smallmatrix} \big) = \lambda (1, \ell_{\lambda})
\big( \begin{smallmatrix} \alpha \\   \beta \end{smallmatrix} \big) 
= \lambda (\alpha+ \beta \ell_{\lambda}). \]
The same holds for $\ell _{\lambda}^{\star}$.
In the next paragraph, the important idea is that $M_{\sigma}$ acts
thus as an automorphism in $\Z[\ell]$, and that a stepped path
considered here is in one-to-one correspondence with a tiling of the
line. 

We proceed by translating the formal sum 
$$E_1^*(\sigma)(W,i^*)=\sum_{j,k :\  \sigma(j)[k]=i} \big(
M_{\sigma}^{-1}(W+A(\sigma(j)_{k+1}\cdots \sigma(j)_{|\sigma(j)|}), j^{\ast} \big)$$
into the language of digit sets and tile-substitutions, and into the
internal space $H$. 
Then, multiplication by $M_{\sigma}^{-1}$ in the integers
$\Z[\ell_{\lambda}]$, embedded in $G$, is just multiplication by
$\lambda^{-1}$ in $G$. And, by construction of the lattice
$\Lambda$, multiplication by $M_{\sigma}^{-1}$ in the integers
$\Z[\ell_{\lambda}]$, embedded in $H$, is multiplication by
$(\lambda')^{-1}= (\lambda^{-1})^{-1} = \lambda$ in $H$ (we use the fact that  $\sigma$ has determinant $1$). Furthermore, 
the term $M^{-1}_{\sigma} A (u)$, projected to $H$, reads in
$\Z[\ell_{\lambda}]$ as $A(u) = |u|_a + |u|_b \ell_{\lambda}^{\star}$. The formal sum
above translates into 

\[ x + T^{\star}_i  \mapsto \{ \lambda x + T^{\star}_j - t_{ijn} \mid
 n,j, \mbox{ such that the } n\mbox{-th letter in } \sigma(j)
\mbox{ is } i \}, \]

where $t_{ijn}$ denotes the `prefix' of the $n$-th letter, this
means here: the digit $d_{ijn}$ starred, that is, $d^{\star}_{jin}$.   Here $x$ is the projection of $W$.
In other words, this means

\[ x + T^{\star}_i  \mapsto \lambda x + \{ T^{\star}_j -
d^{\star}_{jin} \mid  i=1, \ldots, m, n=1, \ldots, N \} = \lambda x + \{ T^{\star}_j
- \D^{\star}_{ji} \mid  i=1, \ldots, m \}, \]

where $N = N(i,j) = | \{ j \, | \, \sigma(j)_n = i \}$. This shows that $s^{\star}$ and $E_1^{\ast}(\sigma)$ are equivalent.

The equivalence $\Xc_{\overline{\sigma}}\cong \Xc_{\sigma^{\ast}}$
comes from the fact that $\overline{\sigma}$ and $\sigma^{\ast}$ have
equivalent  substitution  matrices in $SL(2, {\mathbb Z})$. Indeed, 
 the substitution  matrix of $\sigma^{\ast}$ is the transpose
of $M_{\sigma}$, whereas 
 the substitution  matrix of ${\overline{\sigma}}$
satisfies $M_EM_{\overline{\sigma}}M_E = M_{\sigma } ^T$ according to Remark \ref{rem:factor}.
By  Theorem
\ref{thm:sturmconjug} and \ref{thm:rigidity} the claim follows.
\end{proof}

\section{Selfduality} \label{sec:selfdual}

With respect to Definition \ref{def:invsubs}, it is now natural to ask
which substitutions and  substitution hulls are selfdual.

\begin{defi}\label{def:selfdual}
Let $\sigma$ be a primitive two-letter substitution over $\{a,b\}$ with determinant $1$. If $\Xc_{\sigma}
= \Xc_{\overline{\sigma}}$, or $\Xc_{\sigma} =
E(\Xc_{\overline{\sigma}})$,  then $\sigma$ is  said to have a {\em selfdual hull}.

If $\sigma  \sim  \overline{ \sigma}$ or  $\sigma  \sim  E\overline{ \sigma}E$,    then the substitution   $\sigma$ is  said to be  a   {\em selfdual substitution}.
\end{defi}
By Theorem \ref{thm:inv}, this definition translates immediately to
star-duals $s^{\star}$, and to dual maps of substitution
$\sigma^{\ast}$. 

According to  Remark \ref{rem:factor} and Corollary \ref{cor:xsxr},  we deduce the following:

\begin{prop} \label{prop:selfsim}
A substitution is  selfdual if and only if its hull  $\Xc_{\sigma}$ is  selfdual.
\end{prop}

\begin{prop} \label{pmp}
If $\sigma$ is selfdual (with determinant $1$),  then $(M_{\sigma})^{-1} = Q^{-1}M_{\sigma}Q$
for either $Q = M_{\tau_a}=\big( \begin{smallmatrix} -1 & 0 \\ 0 &
  1 \end{smallmatrix} \big)$, or $Q =M_{\tau_a}M_ E= \big( \begin{smallmatrix} 0 & -1
  \\ 1 & 0 \end{smallmatrix} \big)$; and $M_{\sigma} ^T  = P 
M_{{\sigma}}  P ^T$, with either $P=M_E= \big( \begin{smallmatrix} 0
  & 1 \\ 1 & 0 \end{smallmatrix} \big)$, or $P = \id$. 
\end{prop}
\begin{proof}
If  $\sigma$ is selfdual, then $M_{\sigma} =
M_{\overline{\sigma}}$, or $M_{\sigma} = M_E M_{\overline{\sigma}} M_E$.
Furthermore,  by Remark \ref{rem:factor},
 $M_EM_{\overline{\sigma}}M_E = M_{\sigma } ^T.$
 Hence  one has either  $M_{\sigma} ^T  =  
M_{{\sigma}}  $, or  $M_{\sigma} ^T  =   EM_{\sigma} E$.

Provided that $\mbox{det} M_{\sigma}=+1$,  one has
$M_{\sigma}^{-1} =\big( \begin{smallmatrix} s & -q  \\ -r  & p  \end{smallmatrix}
\big), $ with 
 $M_{\sigma}=\big( \begin{smallmatrix} p & q  \\ r  & s  \end{smallmatrix}
\big)$.

If   $M_{\sigma} = M_{\sigma}  ^T
$,  then 
one gets $
 (M_{\tau_a} M_E)^{-1}   M_{\sigma}  (    M_{\tau_a}  M_E) =  \big( \begin{smallmatrix} s & -r  \\ -q & p  \end{smallmatrix}
\big)= M_{\sigma^{-1}} ,
$
since $ r=q  $.

If   $M_{\sigma} = M_E M_{\sigma}  ^T
M_E= \big( \begin{smallmatrix} p & r  \\ q & s  \end{smallmatrix}
\big) $,  then 
$  M_{\tau_a}    M_{\sigma}   M_{\tau_a}  =    \big( \begin{smallmatrix} p & -q  \\ -r & s  \end{smallmatrix}
\big), 
$
since $p=s$.
\end{proof}

The following theorem states that this necessary condition is already
sufficient. 

\begin{thm} \label{thm:master}
Let $P, Q$ be as above (two possibilities each). If $\sigma$ is a
two-letter primitive invertible substitution with $\det M_{\sigma}=1$,
then the following are equivalent:
\begin{enumerate}
\item $\sigma$ is selfdual;
\item $M_{\sigma}$ is of the form
\[ M_{m,k} = \big( \begin{smallmatrix} m & k \\ \frac{m^2-1}{k} & m
\end{smallmatrix} \big) \quad \mbox{or} \quad M'_{m,k} = \big(
\begin{smallmatrix} m & k \\ k & \frac{k^2+1}{m} 
\end{smallmatrix} \big), \]
where $k \ge 1$ divides $m^2-1$, respectively $m \ge 1$ divides
$k^2+1$;
\item $Q^{-1} M_{\sigma} Q = (M_{\sigma})^{-1}$;
\item $P^T M_{\sigma} P = (M_{\sigma})^T$.

\end{enumerate} 
\end{thm}

\begin{proof} 
It is easily seen  that $(2) \Leftrightarrow (3) \Leftrightarrow (4)$
by simple computation. Indeed,  compute the inverse matrices:
$M_{m,k}^{-1} = \big( \begin{smallmatrix} m & -\frac{m^2-1}{k} \\ -k & m
\end{smallmatrix} \big)$ (resp.\ $(M'_{m,k})^{-1}= \big( \begin{smallmatrix}
  \frac{k^2+1}{m} & -k \\ -k & m  \end{smallmatrix} \big)$). These are
obtained as $Q^{-1}M_{m,k}Q$. Obviously, $Q^T = Q^{-1}$. 

We know already that $(1) \Rightarrow (3)$. It remains to 
prove that $(3) \Rightarrow (1)$. Let $\sigma$ be invertible
with  substitution matrix of the form  (2) in Theorem \ref{thm:master}.
Then, by Proposition \ref{prop:selfsim}, $M_{\overline{\sigma}} =
M_{\sigma}$, or $M_{\overline{\sigma}} = P M_{\sigma}P$. Then, by
Theorem \ref{thm:sturmconjug}, $\sigma$ is conjugate to 
$\overline{\sigma}$, or to $E(\overline{\sigma})$,
proving the claim.  
\end{proof}

\begin{rem}
By Theorem \ref{thm:inv} the above result immediately transfers to
selfdual tile-substitutions (on the line with two tiles) and to dual
maps of substitutions.
\end{rem}

\begin{rem}
For every  matrix $M_{m,k}$ there exists a  selfdual   substitution having this matrix as  substitution matrix,
since the map $\sigma \mapsto M_{\sigma}$  is  onto from the set of  invertible two-letter substitutions onto
$GL(2,\mathbb{Z})$.
\end{rem}

The following result gives a necessary condition not only for $\sigma$
being selfdual, but --- slightly stronger --- for $\sigma \sim
\sigma^*$, in terms of the continued fraction expansion of the
frequency $\alpha$.  Note that  one  can state  analogous conditions for    $\alpha=1-\alpha^{\ast}$.

\begin{thm}\label{thm:sdfreq}
Let $\sigma$   be  a  primitive invertible  substitution on two letters with frequency $\alpha$. Let $\alpha'$
stand for the   algebraic conjugate of $\alpha$.
The following conditions   are equivalent:
\begin{enumerate}
\item
 $\alpha=\alpha^{\ast}$
\item
$2 \alpha \alpha'=\alpha+\alpha'-1$
\item
$\alpha=[0;1+n_1,\overline{n_{2},\cdots ,n_k,n_1}]$  or    $\alpha=[0;1,\overline{n_{1},\cdots ,n_k}]$,
with  the  word  $n_1\cdots n_k$      being a palindrome, i.e.,
 $(n_1,\cdots, n_k)=(n_k,\cdots, n_1)$.
\end{enumerate}
\end{thm}

\begin{proof}
Assertion (2)  comes from Theorem \ref{thm:conjugue}.
Assertion (3)  comes  (2) together  with the known relation between the continued fraction expansion  of  $\alpha$ and its conjugate $\alpha'$ for $\alpha$ being quadratic (see   \cite{BS} Theorem \ref{thm:cfdual}. (see Appendix \ref{sec:arithm} below).
\end{proof}

%The description of  the continued fraction  expansion of $\alpha$  given in Theorem \ref{thm:cfdual}
%is due to   \cite{CM};  such an $\alpha$ is a called  a {\em Sturm  number}; for more details, see  \cite{Allau}.
%See also in the same flavour \cite{bor}. We follow  here the proof of  Theorem 3.7    of \cite{BS} (see  also  the   proof  of  Theorem 2.3.25   of  \cite{Lot2}).

\begin{appendix}

\section{Sturmian words} \label{subsec:sturm}
{\it Sturmian words} are infinite words over a binary alphabet that
have exactly $n+1$ subwords  of length $n$ for every positive integer
$n$. Sturmian words can  also be defined in a constructive way 
 as follows.
Let $\alpha \in (0,1)$. Let ${\mathbb T}^1=\R/\Z$ denote
 the one-dimensional torus. The rotation of angle $\alpha$ of
 ${\mathbb T}^1$ is defined by  $R_{\alpha}: {\mathbb T}^1 \to
{\mathbb T}^1, \ x \mapsto x+\alpha$. 
  For a  given real number  $\alpha$, we introduce the following  two
 partitions of ${\mathbb T}^1$:
\[\underline{I}_a=[0,1-\alpha), \ \ \underline{I}_b=[1-\alpha,1);\ \
\overline{I}_a=(0,1-\alpha], \ \ \overline{I}_b=(1-\alpha,1].\]
Tracing the one-sided (resp. two-sided) orbit of $R^n_\alpha(\rho)$, we   define   two
  infinite (resp.  bi-infinite) words for $\rho \in {\mathbb T}^1$: 
\[\underline{s}_{\alpha,\rho}(n)=\left \{ \begin{array}{l}
a \ \text{ if } R^n_\alpha(\rho)\in \underline{I}_1,\\
b \ \text{ if } R^n_\alpha(\rho)\in \underline{I}_2,\end{array} \right .\]
\[\overline{s}_{\alpha,\rho}(n)=\left \{ \begin{array}{l}
a \ \text{ if } R^n_\alpha(\rho)\in \overline{I}_1,\\
b \ \text{ if } R^n_\alpha(\rho)\in \overline{I}_2.\end{array} \right .\]

It is well known  (\cite{coven,morse}) that an infinite word is a
Sturmian word if and only if it is equal  either to 
$\overline{s}_{\alpha,\rho}$ or  to $\underline{s}_{\alpha,\rho}$ for  
some irrational number $\alpha$.

The notation $c_{\alpha}$ stands in all
 that follows   for
 $\overline{s}_{\alpha,\alpha}=\underline{s}_{\alpha,\alpha}$. This 
 particular Sturmian   word  is called  {\em characteristic  word}.
 
 According to  Theorem \ref{thm:intro},  invertible   word substitutions  on two letters are also  called {\em Sturmian  substitutions}.
 We recall that the set of Sturmian  substitutions 
 is a monoid, with one set of generators  being $\{E,L,\tilde{L}\}$
 \cite{wen,Lot2}.
   Moreover,   we will  use in  Appendix \ref{sec:arithm} the following facts (see \cite{Seebold98} and \cite{Lot2}):
   \begin{prop}\label{prop:car}
      Let $\sigma$ be a   primitive Sturmian substitution with frequency $\alpha$.  There exists  a Sturmian  substitution $\tau$ that is conjugate to $\sigma$  and  satisfies
 $\tau(c_{\alpha})=c_{\alpha}$.

  If $\sigma$  is a  Sturmian substitution that fixes 
  some  characteristic word, then 
  $\sigma$ can be decomposed over the monoid  generated  by   $\{E,L\}$.
  Conversely,   if  $\sigma$  is a  primitive  word substitution that  can be decomposed over the monoid  generated  by   $\{E,L\}$,
  then $\sigma$ is    Sturmian and satisfies
  $\sigma(c_{\alpha})=c_{\alpha}$, with  $\alpha$ being its  frequency.

  \end{prop}

Note that frequently Sturmian words are defined as one-sided infinite
words. For our purposes it is rather natural to consider bi-infinite
Sturmian words. 

\begin{example}
The Fibonacci sequences, that is, the   sequences  of the  hull   $\Xc_{\sigma}$ of   the Fibonacci substitution $\sigma$
(see $\eqref{eq:fib1}$) are Sturmian words with  frequency parameter $\alpha =
\frac{\sqrt{5}-1}{2}$. 
\end{example}
A detailed   description of Sturmian words can be found in  Chapter 2
of \cite {Lot2}, see also Chapter  6  in  \cite{fogg}.

\section{Arithmetic  duality} \label{sec:arithm}
Let us  now   express   the  notion of   duality in terms  of
continued fraction  expansion.  
\begin{thm} \label{thm:cfdual}
Let  $\sigma$ be a   primitive invertible   substitution  over $\{a,b\}$ whose  substitution matrix has  determinant $1$.
The continued  fraction   expansion  of   the   dual   frequency $\alpha^{\ast}$  of  $\alpha$  satisfies the following:

\begin{enumerate}
\item 
if  $\alpha<1/2$,  then    $\alpha=[0;1+n_{1},\overline{n_2,\cdots, n_k, n_{k+1}+n_1}]$,  with $n_{k+1} 
\geq 0$   and $n_1\geq 1$
 \begin{itemize}
 \item
  if $n_{k+1}\geq 1$, then     $\alpha^{\ast}=[0;1,n_{k+1},\overline{n_k,\cdots, n_2, n_1+n_{k+1}}]$ with $k$ is  even
    \item
otherwise,     $\alpha ^{\ast}= [0;1+n_k,\overline{n_{k-1},\cdots, n_2, n_1,n_k}]$  with $k$ odd

\end{itemize}

\item
if  $\alpha  >1/2$,  then  $\alpha=[0;1,n_2,\overline{n_3,\cdots, n_{k-1}, n_{k}+n_2}]$,   with $n_{k}  \geq 0$,   $\alpha ^{\ast}= [0;1+n_{k},\overline{n_{k-1},\cdots, n_3, n_2+n_{k}}]$,  and $k$ is even if $n_k \neq 0$, $k$ is odd  otherwise. 
\end{enumerate}

\end{thm}
The description of  the continued fraction  expansion of $\alpha$  given in Theorem \ref{thm:cfdual}
is due to   \cite{CM};  such an $\alpha$ is a called  a {\em Sturm  number}; for more details, see  \cite{Allau}.
See also in the same flavour \cite{bor}.

\begin{proof}
We follow  here the proof of  Theorem 3.7    of \cite{BS} (see  also  the   proof  of  Theorem 2.3.25   of  \cite{Lot2}).

Let  $c_{\alpha}$  (resp.  $c_{\alpha}^{\ast}$) stand for  the characteristic word of  frequency $\alpha$ (resp. $ \alpha^{\ast}$), such as defined
in Appendix  \ref{subsec:sturm}.
Without loss of generality (by possibly taking a conjugate of $\sigma$ by Proposition \ref{prop:car}),  we  can assume 
  \begin{equation}\label{eq:car}
  \sigma (c_{\alpha})=c_{\alpha}.
  \end{equation}
  
Hence, according  to Proposition  \ref{prop:car},  $\sigma$   can be decomposed   over the monoid $\{E,L\}$, i.e., 
\[\sigma=L^{n_1}  E  L^{n_2}   \cdots  E   L^{n_{k +1}},\]
with $k\geq 1$, $n_1, n_{k+1} \geq 0$, $n_2, \cdots,n_k \geq1$.
Note  that  the parity of  $k$   yields  the   sign of the determinant of  $M_{\sigma}$.

Furthermore, one has $L^{\ast}=R$ and $E^{\ast}=E$ (see Example \ref{ex:dual2}).
Hence the second assertion of Theorem \ref{thm:defdual}  yields 
 \[\sigma^{\ast}=R^{n_{k+1}} E \cdots  E R^{n_1}.\]
 Moreover
 $$ \sigma^{\ast}=E L^{n_{k+1}} \cdots  E L^{n_1}E$$
by using   $R\circ E=E\circ L$.
One deduces  from Proposition \ref{prop:car}  that $ \sigma^{\ast}  (c_{\alpha^{\ast}})=c_{\alpha^{\ast}}.$

  Let us   translate  (\ref{eq:car}) on the continued  fraction expansion of $\alpha$.
  For $m\geq 1$, let $$\theta_m:=L^{m-1}EL.$$   One checks that for every $\beta \in (0,1)$
$$\theta_m (c_{\beta})=c_{1/(m+\alpha)} \mbox{ and } 
 G(c_{\beta})=c_{\frac{1}{1+1/\beta}}.$$

We thus  obtain that if 
\begin{equation}\label{eq:thetater}
c_{\beta}= \theta_{m_1}  \circ  \theta_{m_2} \circ \cdots \circ \theta_{m_k}  \circ  L ^{m_{k+1}}(c_{\gamma}),
\end{equation}
with $k,m_1, \cdots, m_k \geq 1$,   $ 0 <\beta <1$, $0 < \gamma <1$,   $\gamma=[0; \ell_1, \ell_2, \cdots, \ell_i, \cdots]$, 
then  $$\beta=[0;m_1,m_2, \cdots,m_k, m_{k+1}-1+ \ell_1, \ell_2, \cdots, \ell_i ,\cdots].$$

With the previous  notation, 
$$\sigma= \theta_{n_1+1}   \circ  \theta_{n_2} \circ \cdots \circ \theta_{n_k}  \circ  L ^{n_{k+1}-1}.$$
We  distinguish  several cases according to the values of $n_1$ and $n_{k+1}$.
\begin{itemize}
\item We first  assume  $n_1>0$.  Hence   by (\ref{eq:thetater}),  one has  $\alpha=[0;1+n_1, \overline{ n_2,\cdots,n_k, n_{k+1}+n_1}]$ if $n_{k+1} >0$.
If  $n_{k+1}=0$,  we use the fact that
$$E \circ \sigma  \circ E= EL^{n_{1}}E \cdots  L^{n_k} =   \theta_1   \theta_{n_{1}}  \cdots \theta_{n_{k-1}} L^{n_k-1}.$$
  Since $\sigma  (c_{\alpha})=c_{\alpha}$, then 
  $E \circ \sigma \circ E (c_{1-\alpha})=c_{1-\alpha}$. We deduce from
  (\ref{eq:thetater})   that 
 \[1-\alpha=[0;1,\overline{n_{1},  \cdots, n_{k-1},{n_k}}  ].\]
 Since 
$$\alpha=\frac{1}{1+\frac{1-\alpha}{\alpha}}=\frac{1}{1+\frac{1-\alpha}{1-(1-\alpha)}}=\frac{1}{1+\frac{1}{\frac{1}{1-\alpha}-1}},$$
  we get 
  \[\alpha=[0;1+n_1,\overline{n_{2},  \cdots, n_{k},{n_1}}  ].\]

 Furthermore, 
  \[E \circ \sigma^{\ast}  \circ E=L^{n_{k+1}}E \cdots  E L^{n_1}=      \theta_{1+n_{k+1}}  \theta_{n_k} \cdots    \theta_{n_2}  L^{n_1-1}.\]

Since $\sigma^{\ast}    (c_{\alpha^{\ast}})=c_{\alpha^{\ast}}$, then 
  $E \circ \sigma^{\ast}  \circ E (c_{1-\alpha^{\ast}})=c_{1-\alpha^{\ast}}$. We deduce from
  (\ref{eq:thetater})   that 
 \[1-\alpha^{\ast}=[0;1+n_{k+1}, \overline{n_k,\cdots,n_1+n_{k+1}}  ].\]
From
$$\alpha^{\ast}=\frac{1}{1+\frac{1}{\frac{1}{1-\alpha^{\ast}}-1}},$$
  we obtain 
  \[\alpha ^{\ast}= [0;1,n_{k+1},\overline{n_k,\cdots, n_2, n_1+n_{k+1}}] \    \text{ if 
  } n_{k+1}>0\]
    and  
  \[\alpha ^{\ast}= [0;1+n_k,\overline{n_{k-1},\cdots, n_2, n_1,n_k}]  \   \text{  if  } n_{k+1}=0.\]

\item We now  assume  $n_1=0$.
One has  \[\sigma^{\ast}=R^{n_{k+1}} E \cdots   R^{n_2} E =E L^{n_{k+1}} \cdots  E L^{n_2}.\]
\begin{itemize}
\item
We   assume $n_{k+1}>0$. One has  $\alpha=[0;1,\overline{n_2,\cdots, n_k, n_{k+1}}]$.
  Moreover
\[\sigma^{\ast}=    \theta_1  \theta_{n_{k+1}}  \cdots    \theta_{n_3}  L^{n_2-1}.\]
We thus get 
\[\alpha^{\ast}=[0;1,\overline{n_{k+1},\cdots  n_3,n_2}].\]

\item We now  assume  $n_{k+1}=0$.   One has
$$E \circ \sigma  \circ E= L^{n_{2}}E \cdots  L^{n_k} =      \theta_{n_{2}+1}  \cdots \theta_{n_{k-1}} L^{n_k-1}.$$
  Since $\sigma  (c_{\alpha})=c_{\alpha}$, then 
  $E \circ \sigma \circ E (c_{1-\alpha})=c_{1-\alpha}$. We deduce from
  (\ref{eq:thetater})   that 
 \[1-\alpha=[0;n_2+1,\overline{n_{3},  \cdots, n_{k-1},{n_k+n_2}}  ].\]
 Since 
$$\alpha=\frac{1}{1+\frac{1}{\frac{1}{1-\alpha}-1}},$$
  we get 
  \[\alpha=[0;1,n_2, \overline{n_{3},  \cdots,  n_{k-1},n_{k}+{n_2}}  ].\]

Furthermore  \[\sigma^{\ast}=ER^{n_{k}} E \cdots   R^{n_2}E=    L^{n_{k}} E \cdots    E L^{n_2}=     \theta_{1+n_{k}}  \cdots    \theta_{n_3}  L^{n_2-1}.\]
We thus get   \[\alpha ^{\ast}= [0;1+n_{k},\overline{n_{k-1},\cdots, n_3, n_2+n_{k}}].\]

\end{itemize}
\end{itemize}
In conclusion,   we have proved that
\begin{enumerate}
\item  $\alpha ^{\ast}= [0;1,n_{k+1},\overline{n_k,\cdots, n_2, n_1+n_{k+1}}]$ with $k$ even
if $\alpha=[0;1+n_{1},\overline{n_2,\cdots, n_k, n_{k+1}+n_1}]$,  with $n_1>0$  and $n_{k+1}>0$;
\item
  $\alpha ^{\ast}= [0;1+n_k,\overline{n_{k-1},\cdots, n_2, n_1,n_k}]$ with $k$ odd
if $\alpha=[0;1+n_1,\overline{n_2,\cdots, n_k,n_1}]$ with $n_1>0$;
\item
$\alpha^{\ast}=[0;1,\overline{n_{k+1},\cdots  n_3,n_2}]$ with $k$ odd
if   $\alpha=[0;1,\overline{n_2,\cdots, n_k, n_{k+1}}]$ with $n_{k+1} >0$;
\item
$\alpha ^{\ast}= [0;1+n_{k},\overline{n_{k-1},\cdots, n_3, n_2+n_{k}}]$ with $k$ even
if  $\alpha=[0;1,n_{2},\overline{n_3,\cdots, n_{k-1}, n_{k}+n_2}]$.
\end{enumerate}
\end{proof}
\bigskip

We now give the proof of Theorem \ref{thm:conjugue}: we  want to prove that 
$$\alpha^{\ast}= \frac{1-\alpha'}{2\alpha'-1}, $$
where $\alpha'$ is the  algebraic conjugate of $\alpha$.
\begin{proof} {\em (of Theorem \ref{thm:conjugue})} 
Our proof is inspired  by    the proof of  Theorem 2.3.26 in  \cite{Lot2}.
Let $\alpha$ be the frequency of a   primitive two-letter substitution $\sigma$.
Without loss of generality, we assume  again $\sigma(c_{\alpha})=c_{\alpha}$.
We will use the fact that if  $\gamma=[a_1;\overline{a_2,\cdots,  a_n, a_1}]$, then $$-1/\gamma'=[a_n;\overline{a_{n-1}, \cdots,a_1,a_n}].$$

\begin{itemize}
\item 
We first assume $\alpha<1/2$. According to Theorem \ref{thm:cfdual},     $\alpha=[0;1+n_{1},\overline{n_2,\cdots, n_k, n_{k+1}+n_1}]$,  with $n_{k+1} 
\geq 0$   and $n_1\geq 1$.  Let  \[\gamma=[n_2;\overline{n_3,\cdots, n_k, n_{k+1}+n_1,n_2}].\]  We have 
$1/\alpha=1+n_1+ 1/\gamma  $  and $1/\alpha'=1+n_1+ 1/\gamma'$.
One has  \[-1/\gamma'=[n_{k+1}+n_1;\overline{n_k,\cdots, n_2, n_{k+1}+n_1}].\]
We deduce  that $-(1/\gamma'  +n_1+n_{k+1})=[0;\overline{n_k, \cdots,n_2,n_1+n_{k+1}}]$.
 \begin{itemize}
 \item
   If $n_{k+1}\geq 1$, then 
  $\alpha^{\ast} =[0;1,n_{k+1},\overline{n_k,\cdots, n_2, n_1+n_{k+1}}]$, by Theorem  \ref{thm:cfdual}.
    We obtain
    \[ \frac{\alpha^{\ast}}{1-\alpha^{\ast}}==\frac{1}{\frac{1}{\alpha^{\ast}}-1}=n_{k+1}-(1/\gamma'  +n_1+n_{k+1})=1-1/\alpha',\]
    and thus  \[\alpha^{\ast}= \frac{\alpha'-1}{2\alpha'-1}.\]
    \item
    If $n_{k+1}=0$, then 
    $\alpha^{\ast}=[0;1+n_k,\overline{n_{k-1},\cdots, n_2, n_1,n_k}].$
One  gets  $-(1/\gamma'  +n_1)=[0;\overline{n_k, \cdots,n_2,n_1}]$.
We   deduce that
\[\frac{1}{1/\alpha^{\ast}-1}=-1/\gamma' -n_1=1-1/\alpha'\]
and thus
\[\alpha^{\ast}= \frac{\alpha'-1}{2\alpha'-1}.\]

\end{itemize}

\item
If  $\alpha  >1/2$,  then  $\alpha=[0;1,n_2,\overline{n_3,\cdots, n_{k-1}, n_{k}+n_2}]$,   with $n_{k}  \geq 0$. Let  \[\gamma=[n_3;\overline{n_4,\cdots, n_{k-1}, n_{k}+n_2, n_3}].\]  We have 
$\frac{\alpha}{1-\alpha}=n_2+ 1/\gamma  $  and $\frac{\alpha'}{1-\alpha'}=n_2+ 1/\gamma'$.
One has  \[-1/\gamma'=[n_{k}+n_2;\overline{n_{k-1},\cdots, n_3, n_k+n_2}].\]
We deduce  that $-(1/\gamma'  +n_k+n_{2})=[0;\overline{n_{k-1}, \cdots,n_3,n_2+n_{k}}]$.

 One has $\alpha ^{\ast}= [0;1+n_{k},\overline{n_{k-1},\cdots, n_3, n_2+n_{k}}]$.
 Hence \[ 1/\alpha^{\ast}=1+n_k -(1/\gamma'  +n_k+n_{2})=1-\frac{\alpha'}{1-\alpha'},\]
 and thus \[\alpha^{\ast}= \frac{\alpha'-1}{2\alpha'-1}.\]

\end{itemize}
\end{proof}

\end{appendix}

\section*{Acknowledgements} We would like to thank warmly both anonymous referees for their
useful comments.  We are in particular greatly indebted to one of  the referees for the  much   simpler proof
of Theorem \ref{thm:conjugue} and for a simplification in the proof of Theorem \ref{thm:rigidity}. DF and VS would like to express their
gratitude to the German Research Council (DFG) within the CRC 701.

\end{document}